\documentclass[12pt]{article}
\usepackage{amsfonts,amsmath,amsthm}
\usepackage{boldline}

\usepackage{xparse}
\usepackage{environ}
\let\amsthmproof\proof
\let\amsthmendproof\endproof

\newif\ifignoreallproof
\ifignoreallproof
  \RenewDocumentEnvironment{proof}{s}
   {\ignoreproof}
   {\endignoreproof}
\else
  \RenewDocumentEnvironment{proof}{s}
   {\IfBooleanTF{#1}
     {\ignoreproof}
     {\amsthmproof}}
   {\IfBooleanTF{#1}
     {\endignoreproof}
     {\amsthmendproof}}
\fi

\NewEnviron{ignoreproof}{}
\usepackage{tikz}

\makeatletter
\renewcommand*\env@matrix[1][*\c@MaxMatrixCols c]{%
  \hskip -\arraycolsep
  \let\@ifnextchar\new@ifnextchar
  \array{#1}}
\makeatother

\usepackage{amssymb}
\usepackage{enumerate}
\usepackage[latin1]{inputenc} 
\usepackage{color}
\usepackage{pifont}
\usepackage{multirow}
\usepackage[
paperwidth=210mm,paperheight=290mm,textwidth=180mm, textheight=265mm,lmargin=15mm,top=10mm,
]{geometry}
\usepackage{algorithm}
\usepackage{algorithmic}
\usepackage{mathdots}

\DeclareMathOperator \rank {rank}

\DeclareMathOperator \maxRank{maxRank}
\DeclareMathOperator \minRank{minRank}
\DeclareMathOperator \rows{rows}
\DeclareMathOperator \cols{cols}

\usepackage{arydshln,leftidx,mathtools}

\newtheorem{theorem}{Theorem}[section]
\newtheorem{proposition}[theorem]{Proposition}
\newtheorem{notation}[theorem]{Notation}

\newtheorem{remark}[theorem]{Remark}
\newtheorem{definition}[theorem]{Definition}

\newtheorem{example}[theorem]{Example}
\newtheorem{lemma}[theorem]{Lemma}

\newenvironment{smat}{\left[\begin{smallmatrix}}{\end{smallmatrix}\right]}

\title{The WST-decomposition for partial matrices  
  \footnote{Supported by the Spanish Ministerio y Tecnología MTM2017-90682-REDT.}}
\author{Alberto Borobia, Roberto Canogar\\
\small Dpto. Matem\'{a}ticas, Universidad Nacional de Educaci\'on a Distancia (UNED), 28040 Madrid, Spain\\
\small e-mail: $aborobia@mat.uned.es$, $rcanogar@mat.uned.es$ }
\date{}

\begin{document}

\maketitle

\begin{abstract}

A partial matrix over a field $\mathbb{F}$ is a matrix whose entries are either an element of  $\mathbb{F}$ or an indeterminate and with each indeterminate only appearing  once. A completion  is an assignment  of values in $\mathbb{F}$ to all indeterminates. Given a partial matrix,  through elementary row operations and column permutation it can be decomposed into a block matrix of the form $\begin{smat}{\bf W} & * & *  \\ 0 & {\bf S} & * \\ 0 & 0 &  {\bf T} \end{smat}$ where ${\bf W}$ is wide (has more columns than rows), ${\bf S}$ is square,   ${\bf T}$ is  tall (has more rows than columns), and  these three blocks have at least one completion with full rank. And importantly, each one of the blocks ${\bf W}$, ${\bf S}$ and ${\bf T}$ is unique up to elementary row operations and column permutation whenever ${\bf S}$ is required to be as large as possible. When this is the case $\begin{smat}{\bf W} & * & *  \\ 0 & {\bf S} & * \\ 0 & 0 &  {\bf T} \end{smat}$ will be called a WST-decomposition. With this decomposition it is trivial to compute maximum rank of a completion of the original partial matrix: $\rows({\bf W})+\rows({\bf S})+\cols({\bf T})$. In fact we introduce the WST-decomposition for a broader class of matrices: the ACI-matrices. 
\end{abstract}

\section{Introduction}

\subsection{Preliminaries}

 The ACI-matrices were introduced in 2010 by Brualdi, Huang and Zhan~\cite{MR2680270} as a generalization of {\bf partial matrices} (matrices whose entries are either a constant or an indeterminate and with each indeterminate only appearing  once). 
% \emph{\color{blue}"Entries of an {\bf ACI-matrix} over a field $\mathbb{F}$ are ¿¿linear combinations of constants of $\mathbb{F}$ and indeterminates??, different entries within a column of an ACI-matrix may involve the same indeterminate, and indeterminates that appear in different columns are different. So partial matrices are a subclass of ACI-matrices."} 
 Let $\mathbb{F}[x_1, \ldots , x_k]$ denote  the set of polynomials in the   indeterminates $x_1,  \ldots , x_k$ with coefficients on a field $\mathbb{F}$. A matrix  over $\mathbb{F}[x_1, \ldots, x_k]$ is  an \textbf{A}ffine \textbf{C}olumn \textbf{I}ndependent matrix or {\bf ACI-matrix} if its entries are polynomials of degree at most one and  no indeterminate appears in two different columns.
A \textbf{completion} of an ACI-matrix $A$ is an assignment  of values in $\mathbb{F}$ to all indeterminates  so that it gives a constant matrix in $\mathbb{F}$.  All definitions and most of the results work for any field $\mathbb{F}$, so we will usually omit in what field we are working on.

\begin{definition} The {\bf Rank} of an ACI-matrix $M$ is the set  of all possible ranks of  completions of $M$. The {\bf maxRank} of $M$ is the maximum  rank  for a completion of $M$. The {\bf minRank} of $M$  is the minimum  rank  for a completion of $M$.  We say that   $M$ is {\bf constantRank} if $\maxRank(M)=\minRank(M)$.  
\end{definition} 
 
The Rank of partial matrices has a substantial literature (see Section 1 of \cite{MR2680270}). The constantRank partial matrices were studied in~\cite{McTigueQuinlan3}, and the constantRank ACI-matrices were studied  in~\cite{BoCa1,BoCa2,BoCa3,MR2680270,MR2775784}.

Multiplying an ACI-matrix by a constant square matrix on the left produces  an ACI-matrix of the same size. If, in addition, the constant matrix is nonsingular then the new ACI-matrix will share the same Rank, minRank and maxRank with the old one. The same happens if we permute the columns of an ACI-matrix. Since we are concerned with the Rank of ACI-matrices, the following definition makes sense.
\begin{definition} 
Let $M$ be an $m\times n$ ACI-matrix. For any  nonsingular constant matrix $R$ of order $m$  and for any permutation matrix $Q$ of order $n$, the ACI-matrix $RMQ$ is said to be {\bf equivalent} to  $M$. We represent this equivalence by $M \sim RMQ$.
\end{definition}

In this work we are interested in the study of the Rank of a given ACI-matrix $M$. In order to do it we will  consider the equivalence class of $M$ so that we can find a representative in the class with an easier structure that, for example, reveal directly its maxRank. This easier structure will be the WST-decomposition of $M$ as we will see in Section~\ref{WST}.  
It is important to point out that this definition of equivalence can not be applied  to partial matrices since  $RMQ$ will not be necessarily a partial matrix, it will be an ACI-matrix. %{\color{red}  With this idea of equivalence we will be able to develop a procedure that can be applied to ACI-matrices. And so in particular to partial matrices.  Moreover, it will not be possible to apply the procedure to partial matrices without considering the wider context of ACI-matrices.}

%Note also that not all ACI-matrices  are  equivalent to some  partial matrix, as for example {\color{red} $\begin{smat} x & 1-y \\ 1+x & y\end{smat}$ over the reals.}

\subsection{Basic definitions}

The relation between the number of rows and the number of columns of  ACI-matrices will play an important role, that is why we introduce the following terminology:
\begin{definition} \label{}
Let $M$ be an  ACI-matrix.
\begin{itemize}
\item {\bf rows($M$)} denotes  the number of rows of $M$.
\item  {\bf cols($M$)}  denotes  the number of columns of $M$.
\item $M$ is \textbf{wide} if $\cols(M)> \rows(M)$.
\item $M$ is \textbf{tall} if $\rows(M)> \cols(M)$.
\item $M$ is \textbf{square} if $\rows(M)= \cols(M)$.
\end{itemize}
For technical reasons we will  consider as ACI-matrices the  ones without rows or/and without columns, namely: (i) the {\bf wide degenerate} ACI-matrix $0\times q$  with $q>0$; (ii) the {\bf tall degenerate} ACI-matrix $p\times 0$  with $p>0$;  and (iii)  the \textbf{square degenerate}  or \textbf{void} ACI-matrix $0\times 0$. %We will refer to the square degenerate also as  \textbf{void}.

%For technical reasons we will  consider as ACI-matrices the degenerate ones without rows or/and without columns, namely: (i) the $0\times q$ degenerate  ACI-matrix $D_{0\times q}$ with $q>0$ will be  wide; (ii) the $p\times 0$  degenerate ACI-matrix $D_{p\times 0}$ with $p>0$ will be  tall;  and (iii)  the $0\times 0$ degenerate  ACI-matrix $D_{0\times 0}$ will be square. We will refer to $D_{0\times q}$ as {\bf wide degenerate}, to $D_{p\times 0}$ as {\bf tall degenerate}, and to $D_{0\times 0}$ as  \textbf{square degenerate} or \textbf{void}.
\end{definition}

A constant matrix $M$ is full row rank if $\rank(M)=\rows(M)$, is full column rank if $\rank(M)=\cols(M)$, and is full rank if $\rank(M)=\min\{\rows(M),\cols(M)\}$. We will adapt this common terminology to the maxRank of ACI-matrices.

 \begin{definition} \label{mfr}
The  ACI-matrix $M$ is 
\begin{itemize}
\item  {\bf Full Row maxRank} or {\bf FRmR} if $\maxRank(M)=\rows(M)$.
\item  {\bf Full Column maxRank}  or {\bf FCmR} if $\maxRank(M)=\cols(M)$.
\item  {\bf Full maxRank} or {\bf FmR} if  $\maxRank(M)=\min \{\rows(M), \cols(M)\}$ or, equivalently, if $M$ has a  completion with full rank.
\end{itemize}
Just to emphasize: $(i)$ FRmR is wide or square FmR; $(ii)$ FCmR is tall or square FmR; $(iii)$ FmR is FRmR or FCmR or both.
Again, for technical reasons we will consider a  tall degenerate   to be FRmR, a  wide degenerate  to be FCmR,  and the void   to be FRmR and FCmR.
\end{definition}

The next proposition shows how to build new FmR  ACI-matrices from known FmR ACI-matrices. Its proof is straightforward.

\begin{proposition} \label{FRmR-FCmR}
Let  $M$ be an ACI-matrix. 
\begin{enumerate}
\item If $M$ is FRmR and $M \sim M'$ then  $M'$ is FRmR.
\item If $M$ is FCmR and $M \sim M'$ then  $M'$ is FCmR.
\item \label{dmp7} If $M=\begin{smat} A & B \\  0 & C \end{smat}$   where $A$ and $C$ are FRmR  then $M$ is FRmR.
\item If $M=\begin{smat} A & B \\  0 & C \end{smat}$   where $A$ and $C$ are FCmR  then $M$ is FCmR.
\item If $M=\begin{smat} A & B \\  0 & C \end{smat}$   where $A$ and $C$ are square FmR  then $M$ is square FmR.
\end{enumerate}
\end{proposition}

\subsection{The main Theorem}
The first important result in ACI-matrices  appeared in the work where they were introduced.
\begin{theorem}(see~\cite[Theorem 3]{MR2680270}) \label{SubmatrizCerosFirst}
Let $M$ be an $m \times n$  ACI-matrix 
and let $\rho$ be an integer such that $0\leq \rho < \min \{m, n\}$.  The following two statements are equivalent:
\begin{enumerate}[(i)]
\item $\maxRank(M)\leq \rho$.
\item For some positive integers $r$ and $s$ with $\rho=(m-r)+(n-s)$ there exist a nonsingular constant matrix $R$ and a permutation matrix $Q$ such that  $RMQ= \begin{smat}A & B \\  0 & C   \end{smat}$ where $0$ is an  $r\times s$  submatrix with all its entries equal to zero.
\end{enumerate}
\end{theorem}
It is important to note that the values of $r$ and $s$ in Theorem~\ref{SubmatrizCerosFirst} are not unique in general, and neither are $A$ and $C$ (see the example below). The inspiration of the present work has been to generalize Theorem~\ref{SubmatrizCerosFirst} to find an \emph{analogous}  decomposition which is unique \emph{in some sense}. In our main theorem, Theorem~\ref{existance&uniqueness}, we will show that any ACI-matrix is equivalent to an ACI-matrix $$\left[\begin{array}{ccc}
{\bf W} & * & *  \\
0 & {\bf S} & * \\ 
0    & 0 &  {\bf T} \\
\end{array}\right] $$
where ${\bf W}$ is  wide FRmR or void, ${\bf S}$ is square FmR or void and ${\bf T}$ is  tall FCmR or void. And importantly, each one of the ACI-matrices ${\bf W}$, ${\bf S}$ and ${\bf T}$ are unique up to equivalence whenever ${\bf S}$ is required to be as large as possible. When this is the case the ACI-matrix $\begin{smat}{\bf W} & * & *  \\ 0 & {\bf S} & * \\ 0 & 0 &  {\bf T} \end{smat}$ is called a {\bf WST-decomposition}. This decomposition even works for FmR matrices (note that Theorem~\ref{SubmatrizCerosFirst} did not), but then at least one of the blocks ${\bf W}$ or ${\bf T}$ become void or degenerate ACI-submatrices.

The WST-decomposition will allow us to restrict the study of some properties of ACI-matrices (and for that matter partial matrices) to the case of FmR ACI-matrices. For instance, in this work we will be focused  on the maxRank  and if we know a WST-decomposition of an ACI-matrix it will be trivial to compute its maxRank: $\rows({\bf W})+\rows({\bf S})+\cols({\bf T})$. So we might ask how to find the WST-decomposition in practice. A work that explains an algorithm that computes efficiently the WST-decomposition is in preparation. The maxRank for partial matrices was treated in~\cite{CJRW} where the authors provide an procedure to compute it.  Our  algorithm will permit us to compute the maxRank for the broader class of ACI-matrices.

\bigskip

{\bf Example:} Below we present a $5\times 5$ ACI-matrix (it is actually a partial matrix) with maxRank 4, and with three different block partitions that verify the condition (ii) of Theorem~\ref{SubmatrizCerosFirst}:
{\small\begin{align} \label{example4x4}
M=\left[\begin{array}{cc|ccc}
1 & x_1 & y_1 & z_1 & 1\\ \hline
0 & 0 & y_2  & z_2 & t_1 \\
0 & 0 & 0 & z_3 & t_2\\
0 & 0 & 0 & 0 & t_3\\
0 & 0 & 0 & 0 & 1
\end{array}\right]
=\left[\begin{array}{ccc|cc}
1 & x_1 & y_1 & z_1 & 1\\
0 & 0 & y_2  & z_2 & t_1 \\ \hline
0 & 0 & 0 & z_3 & t_2\\
0 & 0 & 0 & 0 & t_3\\
0 & 0 & 0 & 0 & 1
\end{array}\right]
=
\left[\begin{array}{cccc|c}
1 & x_1 & y_1 & z_1 & 1\\ 
0 & 0 & y_2  & z_2 & t_1 \\
0 & 0 & 0 & z_3 & t_2\\ \hline
0 & 0 & 0 & 0 & t_3\\ 
0 & 0 & 0 & 0 & 1
\end{array}\right].
\end{align}}

For $M$ a valid WST-decomposition is:
\begin{align*}
{\small\left[\begin{array}{cc|cc|c}
1 & x_1 & * & * & *\\ \hline
0 & 0 & y_2  & z_2 & * \\
0 & 0 & 0 & z_3 & *\\ \hline
0 & 0 & 0 & 0 & t_3\\ 
0 & 0 & 0 & 0 & 1
\end{array}\right]} \text{ where } {\bf W}= \begin{bmatrix}1 & x_1\end{bmatrix},\ {\bf S}= \begin{bmatrix} y_2 & z_2 \\ 0 & z_3\end{bmatrix},\ {\bf T}=\begin{bmatrix}t_3 \\ 1\end{bmatrix}.
\end{align*}

The property of ${\bf S}$ being as big as possible is required for the uniqueness of ${\bf W}$, ${\bf S}$ and ${\bf T}$ up to equivalence. Because decompositions like the following meet all the other requirements
{\small\begin{align*}
\left[\begin{array}{cc|c|cc}
1 & x_1 & * & * & *\\ \hline
0 & 0 & y_2  & * & * \\ \hline
0 & 0 & 0 & z_3 & t_2\\ 
0 & 0 & 0 & 0 & t_3\\ 
0 & 0 & 0 & 0 & 1
\end{array}\right].
\end{align*}}

%\newpage

\section{Zero blocks}

Given an ACI-matrix, we will be interested in finding  equivalent ACI-matrices which have a lot of zeros. A submatrix with all its entries equal to  0 will be referred as a {\bf zero block}. A measure associated to the size of a zero block that we will frequently use is  its number of rows plus its number of columns.

\begin{definition} \label{defBigZero}
Let  $\begin{smat} A & B \\  0 & C \end{smat}$ be an $m\times n$ ACI-matrix where the zero block $0$ is of size $r\times s$.
\begin{itemize}
\item The zero block in $\begin{smat} A & B \\  0 & C \end{smat}$ is  \textbf{Big}  when  $r+s>\max\{ m, n\}$.
\item The zero block in $\begin{smat} A & B \\  0 & C \end{smat}$ is  \textbf{Medium} when  $r+s=\max\{ m, n\}$.
\end{itemize}
\end{definition}
Note that  a Medium zero block measures one less than the smallest Big zero block. Again, for technical reasons we include the possibility for a Medium zero block to be  degenerate. In our next result we provide  equivalent and more intuitive definitions for Big and for Medium zero blocks. 
\begin{proposition} \label{corBig}
For  an  ACI-matrix $M=\begin{smat} A & B \\  0 & C \end{smat}$ with a zero block  the following properties are satisfied: 
\begin{enumerate}[(i)]
%\item The zero block  is Big if and only if  $A$ is  wide and    $C$ is  tall.
\item The zero block is  Big  if and only if $A$ is wide and $C$ is tall. 
\item The zero block  is Medium if and only if either (1) $M$ is tall, $A$ is square and $C$ is tall; (2) $M$ is square, $A$ and $C$ are square; or (3) $M$ is wide, $A$ is wide and $C$  is square.
\end{enumerate}
\end{proposition}
\begin{proof}  Let $M=\begin{smat} A & B \\  0 & C \end{smat}$ be an  $m\times n$ ACI-matrix with a $r\times s$  zero block, that is, 
\begin{align}\label{MDecomp}
M =\left[\begin{array}{cc}
A & B  \\ 
\smash{\underset{s}{\underbrace{0}}}     & \smash{\underset{n-s}{\underbrace{C}}}
\end{array}\right] \hspace{-2mm}\begin{array}{ll}
\} \ m-r \\
\} \ r
\end{array} 
\end{align}

\begin{enumerate}[(i)]
\item The zero block of  $M$ is  Big if and only if 
\begin{align*}
r+s>\max\{ m, n  \}  \Leftrightarrow 
\begin{cases} 
s>m -r \\
r>n -s 
\end{cases}
\Leftrightarrow 
\begin{cases} 
\cols(A)>\rows(A)\\
 \rows(C)>\cols(C)
\end{cases}
\Leftrightarrow 
\begin{cases} 
A \text{ is  wide}\\
C \text{ is  tall}
\end{cases}
\end{align*}

%\item If the zero block of $M$ is Big then 
%$$
%\maxRank\begin{smat} A & B \\  0 & C \end{smat}   \leq 
%\rows(A) +  \cols(C)
% =  (m-r)+(n-s)=(m+n)-(r+s)<\min\{m,n\}
%$$
%and therefore $M$ is not FmR.

\item The zero block of  $M$ is  Medium if and only if 
\begin{align*}
r+s=\max\{ m, n  \}  \Leftrightarrow 
\begin{cases} 
\text{if } M \text{ tall }\;\;\;\;  \begin{cases} 
s=m -r \\
r> n -s 
\end{cases}
\Leftrightarrow 
\begin{cases} 
\cols(A)=\rows(A)\\
 \rows(C)>\cols(C)
\end{cases}
\Leftrightarrow 
\begin{cases} 
A \text{ square}\\
C \text{ tall}
\end{cases} \vspace{3mm}  \\ % 
\text{if } M \text{ square}  \begin{cases} 
s=m -r \\
r= n -s 
\end{cases}
\Leftrightarrow 
\begin{cases} 
\cols(A)=\rows(A)\\
 \rows(C)=\cols(C)
\end{cases}
\Leftrightarrow 
\begin{cases} 
A \text{ square}\\
C \text{ square}
\end{cases} \vspace{3mm}  \\ % 
\text{if } M \text{ wide }\;\; \begin{cases} 
s> m -r \\
r=n -s 
\end{cases}
\Leftrightarrow 
\begin{cases} 
\cols(A)> \rows(A)\\
 \rows(C)=\cols(C)
\end{cases}
\Leftrightarrow 
\begin{cases} 
A \text{ wide}\\
C \text{ square}
\end{cases}
\end{cases}.
\end{align*}
%where the first option corresponds to $M$ tall,   the second   to $M$ square and the third  to $M$ wide.
\end{enumerate}
\end{proof}

Let us see a consequence when an ACI-matrix has a Big zero block.
%When an ACI-matrix has a Big zero block then  it  has no full rank completion.
\begin{proposition} \label{Big->nofFullMaxRank}
An  ACI-matrix  with a  Big  zero block is not FmR. 
\end{proposition}

\begin{proof} Let $M=\begin{smat} A & B \\  0 & C \end{smat}$ be as in  (\ref{MDecomp}).  If the zero block is Big then 
$$
\maxRank\begin{smat} A & B \\  0 & C \end{smat}   \leq 
\rows(A) +  \cols(C)
 =  (m-r)+(n-s)=(m+n)-(r+s)<\min\{m,n\}
$$
and therefore $M$ is not FmR.
\end{proof}

If a Big or Medium zero block is present in an ACI-matrix it makes it trivial to compute the maxRank when the diagonal blocks  are FmR. 

\begin{theorem} \label{Mrank+Mrank=Mrank}
If the ACI-matrix $\begin{smat} A & B \\  0 & C \end{smat}$ has a Big or Medium  zero block then the following conditions are equivalent:
\begin{enumerate}[(i)]
\item $A$ is FRmR and $C$ is FCmR.
\item $\maxRank\begin{smat} A & B \\  0 & C \end{smat}=\rows (A)+\cols (C)$.
\end{enumerate}
\end{theorem}
\begin{proof}  
Since 0 is a Big or Medium zero block of $\begin{smat} A & B \\  0 & C \end{smat}$ then $A$ is wide or square and $C$ is tall or square. 

\begin{description}
\item[$(i) \Rightarrow (ii)$]  Let $\begin{smat} \widehat{A} & \widehat{B} \\  0 & \widehat{C}\end{smat}$  be a completion of $\begin{smat} A & B \\   0 & C \end{smat}$ such that $\widehat{A}$ and $\widehat{C}$ are full rank. Then we have  
\begin{align*}
\maxRank\begin{smat} A & B \\  0 & C \end{smat} \leq 
\rows (A) + \cols (C) =
\rank(\widehat{A})+\rank(\widehat{C}) \leq
\rank\begin{smat} \widehat{A} & \widehat{B} \\  0  & \widehat{C} \end{smat} \leq 
\maxRank\begin{smat} A & B \\  0 & C \end{smat}
\end{align*}
and the result follows.
\item[$(ii) \Rightarrow (i)$] 
 
We divide the proof in two parts:
\begin{enumerate}[(a)]
\item Note that  $\maxRank\begin{smat}B\\ C \end{smat} \leq \cols(C)$ and $\maxRank(A)\leq \rows(A)$, so
$$
\maxRank\begin{smat} A & B \\  0 & C \end{smat} \leq 
\maxRank\begin{smat} A \\  0 \end{smat}+ \maxRank\begin{smat}B\\ C \end{smat} \leq
\rows(A) +  \cols(C).
$$ 
Since   $
\maxRank\begin{smat} A & B \\  0 & C \end{smat} = 
\rows(A) +  \cols(C)
$ then  $\maxRank(A)=\rows(A)$.
So $A$ is FRmR.  
\item Note that  $\maxRank\begin{smat} A & B \end{smat} \leq \rows(A)$ and $\maxRank(C)  \leq \cols(C)$, so 
$$
\maxRank\begin{smat} A & B \\  0 & C \end{smat} \leq 
\maxRank\begin{smat} A & B \end{smat}+ \maxRank\begin{smat}0 &  C \end{smat} \leq
\rows(A) +  \cols(C).
$$ 
Since   $
\maxRank\begin{smat} A & B \\  0 & C \end{smat} = 
\rows(A) +  \cols(C)
$ then  $\maxRank(C)=\cols(C)$.
So $C$ is FCmR.  
\end{enumerate}%\item[$(ii) \Rightarrow (i)$] Suppose that $A$ is not FRmR. Then $\maxRank(A)<\rows(A)$ and 
%$$
%\maxRank\begin{smat} A & B \\  0 & C \end{smat} \leq
%\maxRank\begin{smat} A \\  0 \end{smat}+ \maxRank\begin{smat}B\\ C \end{smat} <
%\rows(A) +  \cols(C).
%$$ 
% Suppose now that $C$ is not FCmR. Then $\maxRank(C)<\cols(C)$ and 
%$$
%\maxRank\begin{smat} A & B \\  0 & C \end{smat} \leq
%\maxRank\begin{smat} A &  B \end{smat}+ \maxRank\begin{smat} 0 &  C \end{smat} <
%\rows(A) +  \cols(C).
%$$ 
%In both cases we achieve a contradiction with the hypothesis.
\end{description}
\end{proof}

\section{Factor and semifactor sets}

Frequently, we will need to permute the columns of an ACI-matrix so that a certain set $F$ of columns appear as the first $\#F$ columns. This will be achieved by right multiplying the ACI-matrix by the appropriate permutation matrix.
\begin{notation}
Let $F=\{f_1,\ldots, f_s\}\subset \{1,\ldots, n\}$ and let $\overline{F}=\{1,\ldots, n\}-\{f_1,\ldots, f_s\}=\{g_1, \ldots,g_{n-s}\}$.
 We define  the permutation $\sigma_F$ of $\{1,\ldots,n\}$  by 
$$
\begin{cases}
 \sigma_F(f_i) =i & \text{ for all } i \in \{1,\ldots,s\}  \\ 
 \sigma_F(g_j)=s+j & \text{ for all } j \in \{1,\ldots,n-s\}
 \end{cases} .
 $$
Note that
$
 \sigma_F(F)= \{1,\ldots, s\}$ and $\sigma_F(\overline{F})= \{s+1,\ldots, n\}
$.
Finally,  ${Q_{F}}$   denotes  the permutation matrix  of order $n$ such that for each $k=1,\ldots,n$ the $k-$th column of any $m \times n$ ACI-matrix $M$    is equal to the $\sigma_F(k)-$th column of $M Q_F$.  \end{notation}

We now introduce two concepts associated to ACI-matrices: factor and semifactor sets. It will be crucial  in this work to determine when an ACI-matrix has factor sets or has semifactor sets, and also to determine the relation between all of its factor sets or between all of its semifactor sets.

\begin{definition}
Let $M$ be an $m\times n$ ACI-matrix. The  set $F\subseteq \{1, \ldots, n\}$ is a  {\bf factor set} of $M$ if there exists a nonsingular matrix $R$ of order $m$  such that \vspace{3mm}
\begin{equation} \label{defDiv27}
R M Q_{F}=\left[\begin{array}{cc} \smash{\overset{\sigma_F(F)}{\overbrace{A}}} &  B\\  0 & C \end{array}\right]
\end{equation}
where  the zero block   is Big,  $A$ is (wide) FRmR  and  $C$ is (tall) FCmR. We will say that   $RMQ_F$ is  an \textbf{$F$-decomposition} of $M$. 
\end{definition}

%{\color{blue} 

Note that in (\ref{defDiv27}) $A$  is wide and $C$ is  tall  since the zero block  is Big. For completeness we  put  in Table~\ref{tabla1} all  cases that are possible in (\ref{defDiv27}) for $R M Q_{F}$ taking into account  when degenerate ACI-submatrices appear. %In the table we omit $R$ when we can take $R$ as the identity of order $m$ and we omit $Q_F$ when we can take $Q_F$ as the identity of order $n$: 
  \begin{table}[!ht] \small
  \centering
\begin{tabular}{|p{40mm}||p{40mm}|p{35mm}|} \hline
 &  $A$  wide non-degenerate \newline and FRmR   & $A$ wide degenerate    \\ \hline\hline
$C$  tall non-degenerate   \newline and FCmR  & $\begin{smat} A & B \\ 0 & C \end{smat}$  & $\begin{smat} 0 & C \end{smat}$ \\ \hline
$C$ tall degenerate    & $\begin{smat} A  \\0  \end{smat}$\newline    $F=\{1, \ldots, n\}$ & $\begin{smat} 0  \end{smat}$ \newline    $F=\{1, \ldots, n\}$ \\ \hline
\end{tabular}
\caption{\label{tabla1}$\protect R M Q_{F}$ for factor sets.}
\end{table}
%\begin{enumerate}
%\item If $C=D_{p\times 0}$ with $p>0$ then 
%$R M Q_{F}=\begin{smat} A  \\0  \end{smat}$ with $A$  wide FRmR and $F=\{1, \ldots, n\}$.
%
%\item If $A=D_{0\times q}$ with $q>0$ then 
%$R M Q_{F}=\begin{smat} 0 &C \end{smat}$ with $C$  tall FCmR.
%
%\item If $A=D_{0\times q}$ with $q>0$ and  $C=D_{p\times 0}$ with $p>0$ then
%$R M Q_{F}=\begin{smat} 0  \end{smat}$ and $F=\{1, \ldots, n\}$.
%\end{enumerate}}

Now we introduce the concept of semifactor sets, where the Medium zero blocks take the same role as the Big zero blocks for factor sets.

\begin{definition} \label{semifactorSetDefinition}
Let $M$ be an $m\times n$ ACI-matrix. The  set $F\subseteq \{1, \ldots, n\}$ is a  {\bf semifactor set} of $M$ if there exists a nonsingular $R$ of order $m$  such that \vspace{3mm}
\begin{equation} \label{defDiv28}
R M Q_{F}=\left[\begin{array}{cc} \smash{\overset{\sigma_F(F)}{\overbrace{A}}} &  B\\  0 & C \end{array}\right]
\end{equation}
where the zero block is Medium,  $A$ is (wide or square) FRmR  and  $C$ is (tall or square) FCmR. Then $RMQ_F$ is called an   $F$\textbf{-semidecomposition} of $M$.

\end{definition}

%{\color{blue}
  Note that in (\ref{defDiv28}) $A$ or/and $C$ are square since the zero block  is Medium. For completeness we  put  in Table~\ref{tabla2} all  cases that are possible in (\ref{defDiv28}) for $R M Q_{F}$ taking into account  when degenerate ACI-submatrices appear. %In the table we omit $R$ when we can take $R$ as the identity of order $m$ and we omit $Q_F$ when we can take $Q_F$ as the identity of order $n$: 
  \begin{table}[!ht] \small
  \centering
\begin{tabular}{|p{30mm}||p{46mm}|p{45mm}|p{43mm}|} \hline
 &  $A$  wide/square \newline non-degenerate  and  FRmR  & $A$ wide  degenerate  & $A$ void   \\ \hline\hline
 $C$  tall/square \newline non-degenerate \newline and  FCmR  & Case 1: $\begin{smat} A & B \\ 0 & C \end{smat}$  & Case 2: $\begin{smat} 0 & C \end{smat}$ \newline  $C$ square  since 0 is Medium
%\newline {($C$   tall  imply  0   Big)} 
& Case 3: $\begin{smat} C \end{smat}$  \newline Tall degenerate zero block \newline $F=\emptyset$  \\ \hline
 $C$ tall  degenerate  & Case 4: $\begin{smat} A  \\0  \end{smat}$\newline  $A$ square  since 0 is Medium
%\newline {($A$   wide  imply 0   Big)}
\newline   $F=\{1, \ldots, n\}$ &  $\begin{smat} 0  \end{smat}$  \newline    NOT possible since 0 is Big & Degenerate \\ \hline
$C$ void   & Case 5: $\begin{smat} A  \end{smat}$ \newline Wide degenerate zero block \newline   $F=\{1, \ldots, n\}$ & Degenerate & Degenerate \\ \hline
\end{tabular}
\caption{\label{tabla2}$\protect R M Q_{F}$ for semifactor sets.}
\end{table}

\newpage

Table~\ref{tabla2} shows that FmR ACI-matrices have at least one semifactor set. Note that if $M$ is FmR then $M$ is wide FRmR or square FmR or tall FCmR. Now if $M$ is  wide/square FRmR then   $\{1, \ldots, n\}$ is a semifactor set since we can always take  $M=A$ with $C$ void and the Medium zero block being  wide degenerate (Case 5 in Table~\ref{tabla2}); and if $M$ is  tall/square FCmR then $\emptyset$ is a semifactor set since we can always take  $M= C$ with $A$ void and  the Medium zero block being  tall degenerate (Case 3 in Table~\ref{tabla2}).
 
In the next result we characterize when an ACI-matrix has a factor or a semifactor set. As we will see, the part corresponding to factor sets is quite related with Theorem~\ref{SubmatrizCerosFirst}.

 \begin{proposition} \label{maxRank<->noFactorSet} 
%{\color{red} Factor and semifactor sets appear in an ACI-matrix according to the following:}
The following assertions about factor and semifactor sets  are true:
\begin{enumerate}
\item[$(i)$] An ACI-matrix has a factor set if and only if it is not FmR.
\item[$(ii)$] An ACI-matrix has a semifactor set if and only if it is FmR.
\end{enumerate}
\end{proposition}
\begin{proof} Let $M$ be an $m\times n$ ACI-matrix. 
\begin{description}
\item[$(i) \Rightarrow$]  If $M$ has  a factor set $F$ then  there exists a nonsingular $R$ such that \vspace{3mm}
\begin{equation*}
R M Q_{F}=\left[\begin{array}{cc} \smash{\overset{\sigma_F(F)}{\overbrace{A}}} &  B\\  0 & C \end{array}\right]
\end{equation*}
where the zero block is Big. So $R M Q_{F}$ is not FmR (Proposition~\ref{Big->nofFullMaxRank}) and  thus $M$ is not FmR.

\item[$\quad \Leftarrow$]  If  $M$  is not FmR then $\maxRank(M)<\min\{m, n\}$ and, by  Theorem~\ref{SubmatrizCerosFirst}, for some positive integers $r$ and $s$ with $\maxRank(M)=(m-r)+(n-s)$ there exist a nonsingular   $R$ and a permutation  $Q$ such that 
$RMQ=\begin{smat} A & B \\  0 & C \end{smat}$
 where 0 is an $r\times s$ zero block. Note that 
$$
 \min\{m,n\} > \maxRank(M)=(m-r)+(n-s) \ \Rightarrow \ r+s > m+n-\min\{m,n\} = \max\{m,n\}
$$
so the zero block of $\begin{smat} A & B \\  0 & C \end{smat}$ is Big.
On the other hand
 \begin{align*}
\maxRank\begin{bmatrix} A & B \\  0 & C \end{bmatrix}=\maxRank(M)= (m-r)+(n-s) = \rows(A)+\cols(C)
\end{align*}
which  implies (see Theorem~\ref{Mrank+Mrank=Mrank}) that $A$ is FRmR and $C$ is FCmR. 
 And so, the existence of a factor set for $M$ is proved.

\item[$(ii) \Rightarrow$] 

 If $M$ has a semifactor set $F$ then  there exists a nonsingular $R$ such that \vspace{3mm}
\begin{equation}\label{defDiv29}
R M Q_{F}=\left[\begin{array}{cc} \smash{\overset{\sigma_F(F)}{\overbrace{A}}} &  B\\  0 & C \end{array}\right]
\end{equation}
where the zero block is Medium, $A$ is FRmR and $B$ is FCmR. 
 By Proposition~\ref{corBig} we know that
$A$ or/and $C$  are square, and so both are FRmR or both are FCmR. This implies (see Proposition~\ref{FRmR-FCmR}) that $R M Q_{F}$ is FRmR or FCmR. So $R M Q_{F}$ is FmR, and then $M$ is FmR.

\item[$\quad \Leftarrow$] If $M$ is FmR then, as we have seen just after  Table~\ref{tabla2}, $M$ has at least one semifactor set. 
\end{description}
\end{proof}

%For clarity we include a diagram  that relates all the concepts that we have introduced:
%\begin{align*} 
%\begin{array}{|ccccc|} \hline
%M \text{ is not FmR}& \Leftrightarrow & M  \text{ has factor sets} & \Leftrightarrow & \begin{matrix} M\sim \begin{smat} A & B \\  0  & C \end{smat} \text{ with $0$ Big},   A  \text{  wide FRmR} \\ \text{ and } C \text{  tall FCmR} \end{matrix}\\ \hline
%M \text{ is FmR} & \Leftrightarrow & M \text{ has semifactor sets} & \Leftrightarrow &\begin{matrix}M\sim \begin{smat} A & B \\  0  & C \end{smat}\text{ with $0$ Medium},   A  \text{  wide or square FRmR} \\ \text{ and } C \text{  tall or square FCmR ($A$ or/and $C$ square)}  \end{matrix} \\ \hline
%\end{array}
%\end{align*}

\section{Linear independent rows}

Let $M$ be  an ACI-matrix. Remember that each column of $M$ has its own indeterminates. Suppose that the first column of $A$ has intedeterminates $x_1, x_2, \ldots, x_i$; that the second column $y_1, y_2, \ldots, y_j$; and so on. Now, let us represent the 
 vector space where  the entries of the first column lie by
$\mathbb{F}+\mathbb{F}x_1+\ldots+\mathbb{F}x_i$, and for the second column $\mathbb{F}+\mathbb{F}y_1+\ldots+\mathbb{F}y_j$, and so on. All these sets are  vector spaces over $\mathbb{F}$. And the row vectors of $M$ are in the vector space
\begin{align} \label{vectorSpaceACIrows}
 \big(\mathbb{F}+\mathbb{F}x_1+\ldots+\mathbb{F}x_i\big)\times \big(\mathbb{F}+\mathbb{F}y_1+\ldots+\mathbb{F}y_j\big)\times \ldots
\end{align}
From now on when we talk about linear independence or linear dependence of the rows of an ACI-matrix $M$, we are talking about the vector space given in~(\ref{vectorSpaceACIrows}).

The next Proposition and Remark expose the relation of ACI-matrices with linear independent rows and  ACI-matrices which are FRmR.  It is important to keep in mind this relation.

\begin{proposition} \label{SC->lir}
An FRmR ACI-matrix has linear independent rows.
\end{proposition}
\begin{proof}
Suppose that $M$ is an $m\times n$ ACI-matrix with linear dependent rows. Then any completion of $M$ is a constant matrix with linear dependent rows whose rank is less than $m$. So $\maxRank (M)<\rows(M)$ and so M is not FRmR. 
\end{proof}

\begin{remark}
The linear independence of the rows of an ACI-matrix does not imply that it is FRmR. For example, over any field the  ACI-matrix
\begin{align*}
 \begin{bmatrix} 1 & 1 & 1 & 1 \\ 1 & 1 & 1 & x \\ 1 & 1 & 1 & y \end{bmatrix}
\end{align*}
has linear independent rows and maxRank  equal to 2. So it is not FRmR. 
\end{remark}

The next  result will be relevant in the proof of a key result: Lemma~\ref{RLItoFRC}. In this somewhat long statement the condition that we want to emphasize is that $A_1$ as well as $A_2$ have linear independent rows. This condition will reappear in Lemma~\ref{RLItoFRC}, and actually it will be a central theme of many proofs of our work.

\begin{lemma}\label{TwoDecompDivisor2}
Consider two ACI-matrices  of size $m \times n$ given by 
\begin{equation*}   
M_{1}=\begin{bmatrix} A_{1} & B_{1} \\  0 & C_{1} \end{bmatrix} \quad \text{and} \quad 
M_{2}=\begin{bmatrix}  A_{2} & B_{2} \\  0 & C_{2} \end{bmatrix}.
\end{equation*} 
where $A_{1}$ and $A_{2}$ have the same number $n_1$ of columns. Let   $R$ be a nonsingular constant matrix of order $m$ and let  $Q$ and $Q'$ be permutation matrices of orders $n_1$ and $(n-n_1)$ respectively, such that 
\begin{equation*} 
\begin{bmatrix} A_{2} & B_{2} \\  0 & C_{2} \end{bmatrix}=R\begin{bmatrix} A_{1} & B_{1} \\  0 & C_{1} \end{bmatrix} \begin{bmatrix} Q & 0 \\  0 & Q' \end{bmatrix}.
\end{equation*} 
If $A_{1}$ as well as $A_{2}$ have linearly independent rows then $A_{1} \sim A_{2}$  and $C_{1} \sim C_{2}$.
\end{lemma}

\begin{proof} By hypothesis $A_1$ and $A_2$ have the same number $n_1$ of columns, but nothing is said about the number of rows.  Nevertheless,  as  
$ R \begin{smat} A_{1}  \\  0  \end{smat} Q=  
 \begin{smat} A_{2}  \\  0  \end{smat}$
where $A_{1}$ and $A_{2}$ have linearly independent rows then  $A_{1}$ and $A_{2}$ also have the same number  $m_1$ of rows. Therefore $A_{1}$ and $A_{2}$ have the same size $m_1\times n_1$. Thus $C_{1}$ and $C_{2}$  also have the same size $(m-m_1)\times (n-n_1)$. Writing  
$$R=\begin{bmatrix}S& T \\  U & V \end{bmatrix}$$
 as a block matrix where  $S$ is $m_1\times m_1$ and   $V$ is  $(m-m_1)\times (m-m_1)$ we have 
\begin{equation*} \label{blockEquiv}
  \begin{bmatrix} A_{2} & B_{2} \\  0 & C_{2} \end{bmatrix}= \begin{bmatrix} S& T \\  U & V \end{bmatrix} \begin{bmatrix} A_{1} & B_{1} \\  0 & C_{1} \end{bmatrix} \begin{bmatrix} Q & 0 \\  0 & Q' \end{bmatrix} = \begin{bmatrix} * & * \\  U A_{1}  Q & * \end{bmatrix}.
\end{equation*}
Since  $UA_1Q=0$,   $A_{1}$ has  linearly independent rows, and $Q$ is a permutation   then $U=0$. So
\begin{equation*} \label{blockEquiv2}
\begin{bmatrix} A_{2} & B_{2} \\  0 & C_{2} \end{bmatrix}=
 \begin{bmatrix} S& T \\  0 & V \end{bmatrix} \begin{bmatrix} A_{1} & B_{1} \\  0 & C_{1} \end{bmatrix} \begin{bmatrix} Q & 0 \\  0 & Q' \end{bmatrix} = \begin{bmatrix} {\bf S}A_{1} Q & * \\  0 & V C_{1} Q' \end{bmatrix}.\end{equation*}
As  $S$ and $V$  are  nonsingular then   $A_{1} \sim A_{2}$  and $C_{1}\sim C_{2}$.
\end{proof}

Let $M$ be an  ACI-matrix. Imagine that we  want to find out if $F$ is a factor or a semifactor set of $M$. It makes sense to try to find an  equivalent ACI-matrix  with as many zero rows as possible in the ACI-submatrix formed by the columns indexed by $F$. An efficient way to do this is the procedure of a sweep from bottom to top that we are going to introduce now.

\begin{definition}
%Given an $m\times n$ ACI-matrix $M$, in general we can transform $M$ in many ways into an equivalent ACI-matrix that has all its nonzero rows linearly independent. There is a particular way that we will call a \textbf{sweep from bottom to top in $M$} that consists of $m-1$ steps. The step $i$  is the following:
Let $M$ be an $m\times n$ ACI-matrix and let $F\subseteq \{1,\ldots,n\}$. A \textbf{sweep from bottom to top in $M$} is a procedure that transforms $M$   into an equivalent ACI-matrix that has all its nonzero rows linearly independent. It   consists of $m-1$ steps and   step $i$  is the following:
\begin{description}
\item[step i:] If it is possible, make the ($m-i$)-th row equal to the zero row by adding  linear combinations of   rows $m-i+1,\dots,m$ below it.
\end{description}
A \textbf{sweep from bottom to top in $M$ with respect to the columns of $F$}  is a procedure  as the previous one  but only requiring to do zeros in the entries allocated in the columns corresponding to $F$.
\end{definition}

\begin{example}
For the field or reals consider the ACI-matrix
\[
M=\begin{bmatrix}
x+2 & 1  & z \\
x+1 & 8y & 3z-5 \\
x & 4y & z-2 \\
1 & 4y & 2z-3
\end{bmatrix}.
\]
If we do a sweep from bottom to top in $M$ then
\[
\begin{bmatrix}[ccc]
x+2 & 1  & z \\
x+1 & 8y & 3z-5 \\
x & 4y & z-2 \\
1 & 4y & 2z-3
\end{bmatrix}
\stackrel{\text{step 2}}{\longrightarrow}
\begin{bmatrix}
x+2 & 1  & z \\
0 & 0 &0 \\
x & 4y & z-2 \\
1 & 4y & 2z-3
\end{bmatrix}
\]
and we have finished with an equivalent ACI-matrix whose  nonzero rows   are linearly independent. 

If we do a sweep from bottom to top in $M$ with respect to  $F=\{ 2 \}$  then 
\[
\begin{bmatrix}[c|c|c]
x+2 & 1  & z \\
x+1 & 8y & 3z-5 \\
x & 4y & z-2 \\
1 & 4y & 2z-3
\end{bmatrix}
\stackrel{\text{step 1}}{\longrightarrow}
\begin{bmatrix}[c|c|c]
x+2 & 1  & z \\
x+1 & 8y & 3z-5 \\
x-1 & 0 & -z+1 \\
1 & 4y & 2z-3
\end{bmatrix}
\stackrel{\text{step 2}}{\longrightarrow}
\begin{bmatrix}[c|c|c]
x+2 & 1  & z \\
x-1 & 0 &-z+1 \\
x-1 & 0 & -z+1 \\
1 & 4y & 2z-t
\end{bmatrix} 
\] 
and we have finished with an equivalent ACI-matrix whose nonzero rows in the second column  are linearly independent.  Note that the linear combinations employed to make zeros in the second column of $M$ were extended to the entire rows of $M$.  
\end{example}

\medskip

The definition of factor (resp. semifactor) set just requires one decomposition to exist. If we know somehow that $F$ is a factor  (resp. semifactor)  set of $M$  and we perform a sweep from bottom to top in $M$ with respect to $F$,  we might ask the following question: Do we always arrive, up to permutation of rows and columns, to an $F$-decomposition (resp. $F$-semidecomposition)? The answer is yes, as we will see in the next result where we assume that the sweep from bottom to top with respect to $F$ has already occurred. Then we only need to permute rows and columns to leave a zero block in the bottom left part.  
%{\color{blue}\begin{lemma}\label{RLItoFRC}
%Let  $F$ be a factor (resp. semifactor)  set for  an ACI-matrix $M$ and suppose that permuting in $M$ its rows  by some permutation  $P$ and its columns by $Q_F$ we obtain  \bigskip
%
%\begin{equation} \label{NAQD5}
% P M Q_F = \left[\begin{array}{cc}  \smash{\overset{\sigma_F(F)}{\overbrace{A}}} & B \\  0 & C \end{array}\right]. 
%\end{equation} 
%If $A$  has  linearly independent rows then~(\ref{NAQD5}) is an $F-$decomposition (resp. $F$-semidecomposition). 
%\end{lemma}}

\begin{lemma}\label{RLItoFRC}
Let $M$ be an $m\times n$ ACI-matrix.  Suppose that   $F$ is a factor (resp. semifactor) set of  $M$ and  $P$ is a permutation matrix of order $m$ such that  \bigskip

\begin{equation} \label{NAQD5}
 P M Q_F = \left[\begin{array}{cc}  \smash{\overset{\sigma_F(F)}{\overbrace{A}}} & B \\  0 & C \end{array}\right]
\end{equation} 
with  $A$  having  linearly independent rows. Then~(\ref{NAQD5}) is an $F$-decomposition (resp. $F$-semidecomposition). That is:   the zero block is Big (resp. Medium), $A$ is FRmR and $C$ is FCmR.
\end{lemma}

\begin{proof} 
Note that $PMQ_F$ is obtained from $M$ by permuting its rows by $P$ and its columns by $Q_F$.

 As $F$ is a  factor (resp. semifactor) set of $M$ then there exists a nonsingular $R$ such that \bigskip

\begin{equation} \label{taqs}
 R M Q_F= \begin{bmatrix} \smash{\overset{\sigma_F(F)}{\overbrace{A'}}} & B' \\  0 & C' \end{bmatrix}
\end{equation}
is an $F$-decomposition (resp. $F$-semidecomposition). Therefore   $A'$ is FRmR and so, by Proposition~\ref{SC->lir}, it has all its rows linearly independent. Note that
\begin{equation*} \label{}
\begin{bmatrix} A' & B' \\  0 & C' \end{bmatrix}=R M Q_F=
R\left( P^{-1} \left[\begin{array}{cc}  A & B \\  0 & C \end{array}\right]  Q_F^{-1}  \right)Q_F=
R P^{-1}\left[\begin{array}{cc}  A & B \\  0 & C \end{array}\right] .
\end{equation*}
From Lemma~\ref{TwoDecompDivisor2} we conclude that $A\sim A'$ and $C\sim C'$. 
As $A'$ and $C'$ are FmR then $A$ and $C$ are also FmR. And since the zero block of (\ref{taqs}) is Big (resp. Medium) then the zero block of (\ref{NAQD5}) will also be Big (resp. Medium). So $PMQ_F$ is an $F$-decomposition (resp. $F$-semidecomposition).
\end{proof}

\section{The union and intersection of factors and of semifactor sets}

Most of the heavy lifting of the main result is done in this section.
  In the following three results we will study the relative position of two factor sets or of two semifactor sets of an   ACI-matrix. It is important to recall (see Proposition~\ref{maxRank<->noFactorSet}) that  an ACI-matrix has a factor set if and only if it is not FmR, and that an ACI-matrix has a semifactor set if and only if it is  FmR.

\begin{lemma} \label{constant-nonFullRank.2factorSetsCantBeDisjoint}
Two factor sets of an ACI-matrix
can not be disjoint.
\end{lemma}
\begin{proof} 

Suppose  $F_1$ and $F_2$ are two disjoint factor sets of an $m\times n$ ACI-matrix $M$.  As the empty set is not a factor set of any ACI-matrix then, up to permutation of columns, we can assume that 
$F_1=\{1,\ldots,h\}$ and $F_2=\{h+1,\ldots,k\}$     
with $1\leq h < k \leq n$.  And let $U=\{1,\ldots, n\}\setminus (F_1\cup F_2)$. 

First we do a sweep from bottom to top in $M$ with respect to  $F_1$.  After reordering the rows we obtain \bigskip
\begin{align*} 
M' =\left[\begin{array}{c|cc}
\smash{\overset{F_1}{\overbrace{A}}} & \smash{\overset{F_2}{\overbrace{B}}} &  \smash{\overset{U}{\overbrace{C}}} \\ \cline{1-3}
0    & D             &  E 
\end{array}\right] \hspace{-2mm}\begin{array}{rr}
\} \ r \\
\} \ t
\end{array} 
\end{align*}
where  $A$ has linearly independent rows. As $M\sim M'$ and $M'$ is obtained  from $M$ without permuting columns, then $F_1$ and $F_2$ are factor sets of $M'$.

Now we do a sweep from bottom to top in $M'$ with respecto to $F_2$. After reordering the first $r$ rows and the last $t$ rows we   obtain: \bigskip
\begin{align*}
M''=\left[\begin{array}{c|cc}
\smash{\overset{F_1}{\overbrace{A''}}} & \smash{\overset{F_2}{\overbrace{0}}} & \smash{\overset{U}{\overbrace{C''}}} \\
A' & B' & C' \\ \cline{1-3}
0    & D'              &  E'  \\
0    & 0              &  E'' 
\end{array}\right] \hspace{-2mm} 
\begin{array}{rr}
 \}  \ r_1\\
 \}  \ r_2 \\
 \}  \ t_1 \\
 \}  \ t_2
\end{array}  
\end{align*}
with $r_1+r_2=r$,   $t_1+t_2=t$,
 and where  $\begin{smat} B' \\ D' \end{smat}$ has  linearly independent rows. 
As $A\sim \begin{smat} A'' \\ A' \end{smat}$ then also $\begin{smat} A'' \\ A' \end{smat}$ has linear independent rows.
As $M'\sim M''$ and $M''$ is obtained  from $M'$ without permuting columns   then $F_1$ and $F_2$ are factor sets of $M''$. 

Now let us deduce some inequalities that will be key to our analysis.
On one hand we have a zero block corresponding to the factor set $F_1$, it is formed by the two zeros of the first block column of $M''$. As $\begin{smat} A'' \\ A' \end{smat}$ has linear independent rows then this zero block must be Big (see Lemma~\ref{RLItoFRC}). So 
\begin{align} \label{m,n<t_1+t_2+F_1}
 t_1+t_2+\#F_1 > \max\{ m,n \}.
\end{align}
On the other hand we have a zero block corresponding to the factor set $F_2$, it is formed by the two zeros of the second block column of  $M''$. As $\begin{smat} B' \\ D' \end{smat}$ has  linearly independent rows then this zero block must be Big (see Lemma~\ref{RLItoFRC}). So 
\begin{align} \label{m,n<s_2+t_1+F_2}
 r_1+t_2+\#F_2> \max\{ m,n \}.
\end{align}

 Finally, the culprit of the contradiction will be the  
 $(t_1+t_2)\times ( \# F_2+ \# U)$
 ACI-submatrix of $M''$
\begin{align*}
\left[\begin{array}{cc}
D' & E'   \\
 0              &  E'' 
\end{array}\right],
\end{align*}
since we will prove that it is FmR and not FmR at the same time:
\begin{enumerate}
\item $\begin{smat} D' & E' \\ 0 & E'' \end{smat}$ is FmR because  $F_1$ is a factor set of $M''$ and $\begin{smat} A'' \\ A' \end{smat}$ has linearly independent rows    (see Lemma~\ref{RLItoFRC}).  

\item  The zero block of $\begin{smat} D' & E' \\ 0 & E'' \end{smat}$ is Big if $t_2+\#F_2> \max\{\rows \begin{smat} D' & E' \\ 0 & E'' \end{smat}, \cols\begin{smat} D' & E' \\ 0 & E'' \end{smat}\}$. 
Let us see that this inequality is true:
\begin{enumerate}
 \item $t_2+\#F_2>t_1+t_2$. 
 
 From~(\ref{m,n<s_2+t_1+F_2}) it follows that
\begin{align*}
 r_1+t_2+\#F_2 >m=r_1+r_2+t_1+t_2
\end{align*}
which   implies the required inequality.
 
 \item $t_2+\#F_2 > \# F_2+ \# U $. 
 
 From~(\ref{m,n<t_1+t_2+F_1}) it follows that: 
\begin{align} \label{parte1-raz-L41}
t_1+t_2+\#F_1 >m= r_1+r_2+t_1+t_2 \quad \Longrightarrow \quad  
 \#F_1>r_1.
\end{align}
 From~(\ref{m,n<s_2+t_1+F_2}) it follows that: 
\begin{align} \label{parte2-raz-L41}
r_1+t_2+\#F_2>n= \#F_1+\#F_2+\#U.
\end{align}
From~(\ref{parte1-raz-L41}) and (\ref{parte2-raz-L41}) we obtain the required  inequality.
\end{enumerate}
So $\begin{smat} D' & E' \\ 0 & E'' \end{smat}$ is not FmR because  its   zero block  is Big (see Proposition~\ref{Big->nofFullMaxRank}).
\end{enumerate}
\end{proof}

%%%%%%%%%%%%%%%%%%%%%%%%%%%%%%%%%
%%%%%%%%%%%%%%%%%%%%%%%%%%%%%%%%%
%%%%%%%%%%%%%%%%%%%%%%%%%%%%%%%%%

Although the proof of the next Lemma\ref{SemifactorSet-intersection} is very similar to the proof of Lemma~\ref{constant-nonFullRank.2factorSetsCantBeDisjoint} we will include it for clarity because the differences are subtle. 

 \begin{lemma} \label{SemifactorSet-intersection}
Two semifactor sets of  a  wide ACI-matrix
can not be disjoint.
\end{lemma}
\begin{proof}

 Suppose  $F_1$ and $F_2$ are two disjoint semifactor sets of a wide  ACI-matrix $M$ of size $m\times n$.   The first half of the proof of Lemma~\ref{constant-nonFullRank.2factorSetsCantBeDisjoint} is the same as this one with the difference  that in this case we have two semifactor sets and so the zero blocks that will appear will be Medium instead of Big.
So, after the two sweeps from bottom to top in $M$ and reordering the rows we   obtain:

 \begin{align*}
M''=\left[\begin{array}{c|cc}
\smash{\overset{F_1}{\overbrace{A''}}} & \smash{\overset{F_2}{\overbrace{0}}} & \smash{\overset{U}{\overbrace{C''}}} \\
A' & B' & C' \\ \cline{1-3}
0    & D'              &  E'  \\
0    & 0              &  E'' 
\end{array}\right] \hspace{-2mm} 
\begin{array}{rr}
 \}  \ r_1\\
 \}  \ r_2 \\
 \}  \ t_1 \\
 \}  \ t_2
\end{array}  
\end{align*}

Now let us deduce some equalities that will be key to our analysis.
On one hand we have a zero block corresponding to the semifactor set $F_1$, it is formed by the two zeros of the first block column of $M''$. As $\begin{smat} A'' \\ A' \end{smat}$ has linear independent rows then this zero block must be Medium (see Lemma~\ref{RLItoFRC}). So 
\begin{align} \label{m,n=t_1+t_2+F_1}
 t_1+t_2+\#F_1 = \max\{ m,n \}=n.
\end{align}
On the other hand we have a zero block corresponding to the semifactor set $F_2$, it is formed by the two zeros of the second block column of  $M''$. As $\begin{smat} B' \\ D' \end{smat}$ has  linearly independent rows then this zero block must be Medium (see Lemma~\ref{RLItoFRC}). So 
\begin{align} \label{m,n=s_2+t_1+F_2}
 r_1+t_2+\#F_2= \max\{ m,n \}=n.
\end{align}

Again the culprit of the contradiction will be the  
 $(t_1+t_2)\times ( \# F_2+ \# U)$
 ACI-submatrix of $M''$
\begin{align*}
\left[\begin{array}{cc}
D' & E'   \\
 0              &  E'' 
\end{array}\right], 
\end{align*}
since we will prove that it is FmR and not FmR at the same time:
\begin{enumerate}

\item $\begin{smat} D' & E' \\ 0 & E'' \end{smat}$ is FmR because  $F_1$ is a semifactor set of $M''$ and $\begin{smat} A'' \\ A' \end{smat}$ has linearly independent rows    (see Lemma~\ref{RLItoFRC}).  

\item  The zero block of $\begin{smat} D' & E' \\ 0 & E'' \end{smat}$ is Big if $t_2+\#F_2> \max\{\rows \begin{smat} D' & E' \\ 0 & E'' \end{smat}, \cols\begin{smat} D' & E' \\ 0 & E'' \end{smat}\}$.
Let us see that this inequality is true:
\begin{enumerate}
 \item $t_2+\#F_2>t_1+t_2$. 
 
 From~(\ref{m,n=s_2+t_1+F_2}) it follows that
\begin{align*}
 r_1+t_2+\#F_2 =n>m=r_1+r_2+t_1+t_2
\end{align*}
which   implies the required inequality.
 
 \item $t_2+\#F_2 > \# F_2+ \# U $. 
 
 From~(\ref{m,n=t_1+t_2+F_1}) it follows that: 
\begin{align} \label{parte1-raz-L42}
t_1+t_2+\#F_1 =n>m= r_1+r_2+t_1+t_2 \quad \Longrightarrow \quad  
 \#F_1>r_1.
\end{align}
 From~(\ref{m,n=s_2+t_1+F_2}) it follows that: 
\begin{align} \label{parte2-raz-L42}
r_1+t_2+\#F_2=n= \#F_1+\#F_2+\#U.
\end{align}
From~(\ref{parte1-raz-L42}) and (\ref{parte2-raz-L42}) we obtain the required  inequality.
\end{enumerate}
So $\begin{smat} D' & E' \\ 0 & E'' \end{smat}$ is not FmR because  its   zero block  is Big (see Proposition~\ref{Big->nofFullMaxRank}).
\end{enumerate}
\end{proof}

%%%%%%%%%%%%%%%%%%%%%%%%%%%%%%%%%
%%%%%%%%%%%%%%%%%%%%%%%%%%%%%%%%%
%%%%%%%%%%%%%%%%%%%%%%%%%%%%%%%%%

 \begin{lemma} \label{SemifactorsInTall}
Two semifactor sets of  a tall or square   ACI-matrix  can be disjoint or not  disjoint.
\end{lemma}
\begin{proof}
We provide an example for each case. For tall ACI-matrices 
$$
\begin{array}{lcl}  \vspace{-2mm}
\ \  \overset{F_1}{\overbrace{\phantom{x}}} \ \ \overset{F_2}{\overbrace{\phantom{x}}}  & & 
\ \ \ \overset{F_1}{\overbrace{\phantom{xxxxxx}}} \\ \vspace{-3mm}
\begin{bmatrix}
\ \ 1 \ \ & \ \ 0 \ \ & \ \ x \ \ \\
\ \ 0 \ \ & \ \ 1 \ \ & \ \ y \ \ \\
\ \ 0 \ \ & \ \ 0 \ \ & \ \ 1 \ \ \\
\ \ 0 \ \ & \ \ 0 \ \ & \ \ 1 \ \ \\
\end{bmatrix} 
& \text{and} & 
\begin{bmatrix}
\ \ 1 \ \ & \ \ 0 \ \ & \ \ 0 \ \ \\
\ \ x \ \ & \ \ 1 \ \ & \ \ y \ \ \\
\ \ 0 \ \ & \ \ 0 \ \ & \ \ 1 \ \ \\
\ \ 0 \ \ & \ \ 0 \ \ & \ \ 0 \ \ \\
\end{bmatrix} \\ 
& & 
\phantom{xxxxxx}  \underset{F_2}{\underbrace{\phantom{xxxxxx}}}
\end{array}
$$
and for square ACI-matrices it is enough to delete the last row on each one. 
\end{proof}

%Now we will state  the main result of this section 
%\begin{theorem} \label{cor6}
%The intersection and the union of two factor (resp. semifactor) sets of an ACI-matrix are factor (resp. semifactor) sets.
%\end{theorem}
%
%\begin{proof} In  Lemmas~\ref{constant-nonFullRank.2factorSetsCantBeDisjoint},~\ref{SemifactorSet-intersection} and~\ref{SemifactorsInTall}  we saw that the only case in which two factor or semifactor sets can be disjoint is for semifactor sets in a tall or square ACI-matrix. So we will divide the proof Theorem~\ref{cor6} in two parts: $(i)$ Theorem~\ref{2semifactor-disjointcase} below attack the case of  disjoint semifactor sets, and $(ii)$ Theorem~\ref{2non-overlappingCupAndCapFS} below attack the  overlapping factor or semifactor sets. 
%\end{proof}

%In what follows our main objective will be to prove that the intersection and the union of two factor (resp. semifactor) sets  are factor (resp. semifactor) sets. That is what says Theorem~\ref{cor6} below. Note that Lemmas~\ref{constant-nonFullRank.2factorSetsCantBeDisjoint},~\ref{SemifactorSet-intersection} and~\ref{SemifactorsInTall} say that the only case in which two factor or semifactor sets can be disjoint is for semifactor sets in a tall or square ACI-matrix. So in order to achieve our objective we will study separately the case of  disjoint semifactor sets (Theorem~\ref{2semifactor-disjointcase}) and the case of overlapping factor or semifactor sets for arbitrary ACI-matrices (Theorem~\ref{2non-overlappingCupAndCapFS}). 
%
%----------------

In what follows our main objective will be to prove that the intersection and the union of two factor (resp. semifactor) sets  is a factor (resp. semifactor) set. That is what Theorem~\ref{cor6} below says. Note that Lemmas~\ref{constant-nonFullRank.2factorSetsCantBeDisjoint},~\ref{SemifactorSet-intersection} and~\ref{SemifactorsInTall} conclude that two factor or two semifactor sets always overlap, except in the case of some semifactor sets of tall or square ACI-matrices. So in order to achieve our objective we will first study this exceptional case of disjoint semifactor sets (Theorem~\ref{2semifactor-disjointcase}), and then the generic case of overlapping factor or semifactor sets  (Theorem~\ref{2non-overlappingCupAndCapFS}).

%----------------
%
%Now we will attack  the main result of this section  (Theorem~\ref{cor6}): that the intersection and the union of two factor (resp. semifactor) sets  are factor (resp. semifactor) sets. From Lemmas~\ref{constant-nonFullRank.2factorSetsCantBeDisjoint},~\ref{SemifactorSet-intersection} and~\ref{SemifactorsInTall} it follows that the only case in which two factor or semifactor sets can be disjoint is for semifactor sets in a tall or square ACI-matrix. So we will consider first this case (Theorem~\ref{2semifactor-disjointcase}), and after that the case of overlapping  factor or semifactor sets for arbitrary ACI-matrices (Theorem~\ref{2non-overlappingCupAndCapFS}). 

\begin{theorem} \label{2semifactor-disjointcase}
The intersection and the union of two disjoint semifactor sets of a tall or square ACI-matrix are semifactor sets. 
\end{theorem}

\begin{proof} 
Let $F_1$ and $F_2$ be two disjoint semifactor sets of  a tall or square  ACI-matrix $M$ of size $m\times n$. 

Tall and square FmR ACI-matrices are the only ACI-matrices for which the empty set is a semifactor set (see Table~\ref{tabla2}). Then  $F_1\cap F_2=\emptyset$ is a semifactor set of $M$. 
 
The proof for $F_1\cup F_2$ starts again like the proof of Lemma~\ref{constant-nonFullRank.2factorSetsCantBeDisjoint}.  So after the two sweeps from bottom to top in $M$ and reordering the rows we   obtain:

\begin{align} \label{MatrixWellOrderedQWZ11}
M''=\left[\begin{array}{c|cc}
\smash{\overset{F_1}{\overbrace{A''}}} & \smash{\overset{F_2}{\overbrace{0}}} & C'' \\
A' & B' & C' \\ \cline{1-3}
0    & D'              &  E'  \\
0    & 0              &  E'' 
\end{array}\right] \hspace{-2mm} 
\begin{array}{rr}
 \}  \ r_1\\
 \}  \ r_2 \\
 \}  \ t_1 \\
 \}  \ t_2
\end{array} 
\end{align}
Focusing on the $F_2$ semifactor set, by Proposition~\ref{corBig}  we know that 
$\begin{smat} B' \\ D' \end{smat}$  is square and  so \begin{align} \label{t1<F2}
 \#F_2 \geq t_1.
\end{align}
Focusing on the $F_1$ semifactor set, the zero block formed by the two zeros of the first  column of $M''$ is Medium. So according to Definition~\ref{defBigZero}: 
\begin{align} \label{m,n<t_1+t_2+F_1-2}
 t_1+t_2+\#F_1=\max\{ m,n \}= m.
\end{align}
Let $Z$ be   the zero block composed by the two zero blocks of the last  row of  $M''$. From (\ref{t1<F2}) and (\ref{m,n<t_1+t_2+F_1-2}) 
\begin{align} \label{maxM,N<t_2+F_1+F_2}
t_2+\#F_1+\#F_2\geq \max\{ m,n \}= m
\end{align}
and so $Z$ is either a Big  (if  inequality is strict) or a Medium  (if there is equality) zero block in $M''$.
It can not be Big, otherwise $M''$ would not be FmR (Proposition~\ref{Big->nofFullMaxRank}) and this contradicts  that an ACI-matrix has a semifactor set if and only if it is  FmR (Proposition~\ref{maxRank<->noFactorSet}). 
So $Z$ is a Medium zero block in $M''$. So now we know that there must be equality in~(\ref{maxM,N<t_2+F_1+F_2}), which in turn means that there is equality in~(\ref{t1<F2}): $\#F_2 = t_1$. So $D'$ is square, which implies that $B'$ is wide degenerate and therefore $r_2=0$. So~(\ref{MatrixWellOrderedQWZ11}) is simplified into:

\begin{align*}
M''=\left[\begin{array}{c|cc}
\smash{\overset{F_1}{\overbrace{A''}}} & \smash{\overset{F_2}{\overbrace{0}}} & C'' \\ \cline{1-3}
0    & D'              &  E'  \\
0    & 0             &  E'' 
\end{array}\right] \hspace{-2mm} 
\begin{array}{rr}
 \}  \ r_1\\
 \}  \ t_1 \\
 \}  \ t_2
\end{array} . \\[-3.5mm] \nonumber
\end{align*}
To prove that $F_1\cup F_2$ is a semifactor set, apart from $Z$ being a Medium zero block, we still need to prove that:
\begin{enumerate}[i.]
\item $\begin{smat} A'' & 0 \\ 0 & D' \end{smat}$ is FRmR. Since $F_1$ is a semifactor set we know that $\begin{smat} A'' \\ A' \end{smat}= \begin{smat}A''\end{smat}$ is FRmR, and since $F_2$ is a semifactor set we know that $\begin{smat} B' \\ D' \end{smat}= \begin{smat}D'\end{smat}$ is FRmR. Now we apply Proposition~\ref{FRmR-FCmR} to deduce this item.
\item $E''$ is FCmR. Since $F_1$ is a semifactor set we know that $\begin{smat} D' & E' \\ 0 & E'' \end{smat}$ is FCmR. This together with $D'$ being square implies that $E''$ is FCmR.
\end{enumerate}
\end{proof}

For the proof of our next theorem it will be convenient to introduce the notion of complementary of an ACI-submatrix. Let $A$ be an ACI-submatrix  of  an ACI-matrix $M$,   the \textbf{complementary} $\overline{A}$ of $A$ in $M$ is obtained by deleting all the rows and columns of $M$ that are involved in $A$.

\begin{theorem} \label{2non-overlappingCupAndCapFS}
The intersection and the   union of  two overlapping factor (resp. semifactor) sets of an ACI-matrix are  factor  (resp. semifactor) sets.
\end{theorem}
\begin{proof} Let $F_1$ and $F_2$ be two overlapping factor (resp. semifactor) sets of an ACI-matrix $M$.

If $F_1 \subset F_2$ or $F_2 \subset F_1$  the result is trivial. So, without loss of generality we can assume that 
$F_1=\{1,\ldots,k\}$ and $F_2=\{h+1,\ldots,l\}$
with $1\leq h < k < l$. We do a sweep from bottom to top in $M$ with respect to  $F_1$ and after reordering the rows we obtain 
\begin{align} \label{M'}
& \phantom{XXXx} \overset{F_1}{\overbrace{\phantom{XXXx}}} \nonumber \\[-5mm]
& 
M' =\left[\begin{array}{cc|cc}
{A} & B &   C & D\\ \cline{1-4}
0    & 0              &  E & F
\end{array}\right] \hspace{-2mm}\begin{array}{rr}
\} \ r \\
\} \ t
\end{array} \\[-5mm]
 & \phantom{XXXXXl} \underset{F_2}{\underbrace{\phantom{XXXx}}} \nonumber 
\end{align}
where  $\begin{smat} A  & B \end{smat}$ has linearly independent rows and $\begin{smat} E &  F \end{smat} $ is FCmR (see Lemma~\ref{RLItoFRC}).
Now we do a sweep from bottom to top in $M'$ with respect to $F_2$ and after reordering the first $r$ rows and the last $t$ rows we   obtain
\begin{align*} 
& \phantom{XXXX} \overset{F_1}{\overbrace{\phantom{XXXX}}} \nonumber \\[-5mm]
& 
M''=\left[\begin{array}{cc|cc}
A'' & 0 & 0 & D'' \\
A' & B' & C' & D' \\ \cline{1-4}
0    & 0              &  E' & F' \\
0    & 0              &  0 & F''
\end{array}\right] \hspace{-2mm} \begin{array}{rr}
 \}  \ r_1\\
 \}  \ r_2 \\
 \}  \ t_1 \\
 \}  \ t_2
\end{array} \\[-5mm] 
 & \phantom{XXXXXXl} \underset{F_2}{\underbrace{\phantom{XXXx}}} \nonumber 
\end{align*}
 with $r_1+r_2=r$, $t_1+t_2=t$, $r_1, r_2, t_1, t_2\geq 0$, and where 
$\begin{smat}
A'' & 0 \\
A'    & B'
\end{smat}$ and $\begin{smat}
B' & C' \\
0    & E'
\end{smat}$ 
have  linearly independent rows.  Note that the fourth column of $M''$ could be tall degenerate with size $(r_1+r_2+t_1+t_2)\times 0$ while the other three columns will never be tall degenerate since $1\leq h < k < l$.  Let us see that $t_1=0$ is impossible: if this was the case then the third row of $M''$ would be wide degenerate and since $F_1$ is a factor set then $\begin{smat} 0 & F'' \end{smat}$ should be FCmR, but this is impossible since it has columns full of zeros. So, from now on $t_1>0$. The value $r_2$ might be positive or zero and the arguments we will provide hold for both.

Four possibilities appear depending on the values of $r_1$ and $t_2$. 

\begin{itemize}
\item Suppose $r_1>0$ and $t_2>0$.

\begin{enumerate}[i.]

\item \label{iStep} Since $F_2$ is a factor (resp.  semifactor) set  and $\begin{smat}
B' & C' \\
0    & E
\end{smat}$ 
has  linearly independent rows, then  (see Lemma~\ref{RLItoFRC}) its complementary $
\begin{smat}
A'' & D'' \\ 0 & F''
\end{smat}$  
in $M''$  is FCmR. In particular $A''$ is FCmR.

\item \label{ABstep} Since $F_1$ is a factor (resp. semifactor) set and $\begin{smat} A'' & 0 \\  A' & B' \end{smat}$ has linear independent rows, then (see Lemma~\ref{RLItoFRC})  $\begin{smat} A'' & 0 \\  A' & B' \end{smat}$ is FRmR. In particular $A''$ is FRmR.

\item \label{iiiStep} Since $A''$ is FCmR and FRmR at the same time then  $A''$ is a square FmR.

\item \label{4Step} Since $F_1$ is a factor (resp. semifactor) set and $\begin{smat} A'' & 0 \\  A' & B' \end{smat}$ has linear independent rows, then (see Lemma~\ref{RLItoFRC}) 
its complementary $\begin{smat} E' & F' \\ 0 & F''  \end{smat}$ in $M''$ is FCmR.
 In particular $E'$ is FCmR.

\item  \label{5Step} Since  $F_2$ is a factor (resp. semifactor) set and 
$\begin{smat} B' & C' \\ 0    & E' \end{smat}$ 
has linear independent rows then (see Lemma~\ref{RLItoFRC})
$\begin{smat} B' & C' \\ 0    & E' \end{smat}$ 
is FRmR.  In particular $E'$ is FRmR.

\item \label{viStep} Since $E'$ is FCmR and FRmR at the same time then $E'$ is a square FmR.

\item \label{viiStep} Since $\begin{smat}
E' & F' \\ 
0 & F'' 
 \end{smat}$ is FCmR and $E'$ is square   then $F''$ is FCmR.

\item \label{ivStep} Since $\begin{smat} A'' & 0 \\  A' & B'  \end{smat}$ 
is FRmR and $A''$ is square   then $B'$ is FRmR.

\item \label{xiStep} The complementary matrix of $F''$ in $M''$ is
$\overline{F''}=
\begin{smat}
A'' & 0 & 0 \\
A' & B' & C' \\ 
0    & 0              &  E' 
\end{smat}\sim 
\begin{smat}
B' & A' & C' \\  
0 & A'' & 0 \\ 
0    & 0              &  E'
\end{smat}$. 
Since $B'$, $A''$ and $E'$ are FRmR (\ref{ivStep},  \ref{ABstep} and  \ref{5Step}) then Proposition~\ref{FRmR-FCmR} implies that $\overline{F''}$ is also FRmR.

\item \label{viiiStep} The complementary matrix of $B'$ in $M''$ is
$
\overline{B'}=\begin{smat}
A'' & 0 & D'' \\ 
0 & E' & F' \\ 
0 & 0 & F'' 
\end{smat}$.
Since $A''$,  $E'$ and $F''$  are  FCmR (\ref{iStep},  \ref{4Step} and  \ref{viiStep})  then  Proposition~\ref{FRmR-FCmR} implies that $\overline{B'}$ is FCmR.

  \item \label{ixStep} Consider  the zero block $Z$ obtained by
joining together the three zero blocks in the second column of $M''$, and consider the zero block $Z_1$ corresponding to the factor (resp.  semifactor) set $F_1$.
Note that $Z$ has size $(r_1+t_1+t_2)\times \#(F_1\cap F_2)$, and that $Z_1$   has size $(t_1+t_2)\times \#F_1$. The number $\rows+\cols$ for $Z$ and for $Z_1$ is equal since $A''$ is square~(\ref{iiiStep}).  As $Z_1$ is Big (resp.  Medium) by hypothesis and  the number $\rows+\cols$ is what determines if a zero block is Big (resp.  Medium), then $Z$ is Big (resp.  Medium).

 \item  \label{xiiStep}  Consider  the zero block $Z'$ obtained by 
joining together the three zero blocks in the last row of $M''$, and consider again the zero block $Z_1$ corresponding to the factor (resp.  semifactor) set $F_1$. 
Note that $Z'$ has size $t_2 \times (\#F_1 + \#(F_2\setminus F_1))$ and that $Z_1$ has size $(t_1+t_2)\times \#F_1$. The number $\rows+\cols$ for $Z'$ and $Z_1$ is equal since $E'$ is  square~(\ref{viStep}).  As $Z_1$ is Big (resp.  Medium) by hypothesis  then $Z'$ is Big (resp.  Medium).

\item  $F_1 \cap F_2$ is a factor (resp.  semifactor) set of $M''$ since $Z$ is a Big (resp.  Medium) zero block (\ref{ixStep}), $B'$ is FRmR (\ref{ivStep}) and $\overline{B'}$ is FCmR (\ref{viiiStep}).

\item $F_1 \cup F_2$ is a factor (resp.  semifactor) set of $M''$  since $Z'$ is a Big (resp.  Medium) zero block (\ref{xiiStep}), $F''$ is FCmR (\ref{viiStep}) and  $\overline{F''}$ is FRmR (\ref{xiStep}).
\end{enumerate}

\item Suppose $r_1>0$ and $t_2=0$. Then 
\begin{align*} 
& \phantom{XXXX} \overset{F_1}{\overbrace{\phantom{XXXX}}} \nonumber \\[-5mm]
& 
M''=\left[\begin{array}{cc|cc}
A'' & 0 & 0 & D'' \\
A' & B' & C' & D' \\ \cline{1-4}
0    & 0              &  E & F \\
\end{array}\right] \hspace{-2mm} \begin{array}{l}
 \}  \ r_1\\
 \}  \ r_2 \\
 \}  \ t=t_1 \\
\end{array} \\[-5mm]
 & \phantom{XXXXXXl} \underset{F_2}{\underbrace{\phantom{XXXx}}} \nonumber 
\end{align*}
\begin{enumerate}[i.]

\item \label{1Step2} Since $F_2$ is a factor  (resp.  semifactor)  set  and $\begin{smat}
B' & C' \\
0    & E
\end{smat}$ 
has  linearly independent rows, then  (see Lemma~\ref{RLItoFRC}) its complementary $
\begin{smat}
A'' & D'' 
\end{smat}$ in $M''$   
 is FCmR. In particular $A''$ is FCmR.

\item \label{2Step2} Since $F_1$ is a factor  (resp.  semifactor)  set and $\begin{smat} A'' & 0 \\  A' & B' \end{smat}$ has linear independent rows, then (see Lemma~\ref{RLItoFRC})  $\begin{smat} A'' & 0 \\  A' & B' \end{smat}$ is FRmR. In particular $A''$ is FRmR.

\item \label{3Step2} Since $A''$ is FCmR and FRmR at the same time then  $A''$ is a square FmR.

\item \label{4Step2} Since $F_1$ is a factor  (resp.  semifactor) set and $\begin{smat} A'' & 0 \\  A' & B' \end{smat}$ has linear independent rows, then (see Lemma~\ref{RLItoFRC}) 
its complementary $\begin{smat} E & F  \end{smat}$ in $M''$  is FCmR.
 In particular $E$ is FCmR.

\item \label{5Step2} Since  $F_2$ is a factor  (resp.  semifactor) set and 
$\begin{smat} B' & C' \\ 0    & E \end{smat}$ 
has linear independent rows then (see Lemma~\ref{RLItoFRC})
$\begin{smat} B' & C' \\ 0    & E \end{smat}$ 
is FRmR.  In particular $E$ is FRmR.

\item \label{6Step2} Since $E$ is FCmR and FRmR at the same time then $E$ is a square FmR.

\item \label{7Step2} Since $\begin{smat} A'' & 0 \\  A' & B'  \end{smat}$ 
is FRmR and $A''$ is square FmR  then $B'$ is FRmR.

\item \label{8Step2}  In this step the arguments diverge significantly depending on $F_1, F_2$ being factor or semifactor sets, so we consider the cases separately:
\begin{enumerate}[a.]
\item  $F_1$ and $F_2$ are factor sets.

Since $F_1$ is a factor set then the two zero blocks on the last row of $M''$  compose a Big zero block. This implies (see Proposition~\ref{corBig})  that $\begin{smat} E & F \end{smat}$ is  tall, which is imposible since $E$ is square (\ref{6Step2}).

\item  $F_1$ and $F_2$ are semifactor sets. 

Since $F_1$ is a semifactor set then the two zero blocks on the last row of $M''$ compose a Medium zero block. This implies (see Proposition~\ref{corBig}) that $\begin{smat} E & F \end{smat}$ is tall or square, and since $E$
is square (\ref{6Step2}) this forces $F$ to be tall degenerate with size $t_1\times 0$. So  
\begin{align*} 
& \phantom{XXXX} \overset{F_1}{\overbrace{\phantom{XXXX}}} \nonumber \\[-5mm]
& 
M''=\left[\begin{array}{cc|cc}
A'' & 0 & 0  \\
A' & B' & C'  \\ \cline{1-4}
0    & 0              &  E  \\
\end{array}\right] \hspace{-2mm} \begin{array}{ll}
 \}  \ r_1\\
 \}  \ r_2 \\
 \}  \ t=t_1 \\
\end{array} \\[-5mm]
 & \phantom{XXXXXXl} \underset{F_2}{\underbrace{\phantom{XXXx}}} \nonumber 
\end{align*}
\begin{itemize}

\item In $M''$ the complementary  of $B'$  is
$
\overline{B'}=\begin{smat}
A'' & 0  \\ 
0 & E 
\end{smat}$.
Since $A''$ and  $E$  are square FmR (\ref{3Step2} and  \ref{6Step2})  then  Proposition~\ref{FRmR-FCmR} says that $\overline{B'}$ is square FmR. 

\item  Consider the zero block $Z$ obtained by joining together the zero blocks below and above $B'$. By Proposition~\ref{corBig}  $Z$  is  Medium   since $B'$ is wide or square (\ref{7Step2})  and $\overline{B'}$ is square.

\item As $A''$ is square (\ref{3Step2}), $B'$ is wide or square (\ref{7Step2}) and $E$ is square  (\ref{6Step2}) then $M''$ is wide or square. As  $M''$ has semifactor sets then $M''$ is FmR (see Proposition~\ref{maxRank<->noFactorSet}). Moreover, $F_1\cup F_2$  span all columns of $M''$. Recall the discussion after Definition~\ref{semifactorSetDefinition}  where it was explained that in a wide or square FRmR ACI-matrix the  set $\{1,\ldots,n\}$ is a  semifactor set. So $F_1\cup F_2$ is a semifactor set.

\item  $F_1 \cap F_2$ is a semifactor set of $M''$ since $Z$ is a Medium zero block, $B'$ is FRmR (\ref{7Step2}) and $\overline{B'}$ is FCmR (it is square FmR).
\end{itemize}
\end{enumerate}

\end{enumerate}

\item Suppose $r_1=0$ and $t_2>0$. Then  
\begin{align*} 
& \phantom{XXXx} \overset{F_1}{\overbrace{\phantom{XXXx}}} \nonumber \\[-5mm]
& 
M''=\left[\begin{array}{cc|cc}
A & B & C & D \\ \cline{1-4}
0    & 0  &  E' & F' \\
0    & 0  &  0 & F''
\end{array}\right] \hspace{-2mm} \begin{array}{l}
 \}  \ r=r_2 \\
 \}  \ t_1 \\
 \}  \ t_2
\end{array} \\[-5mm]
 & \phantom{XXXXxXl} \underset{F_2}{\underbrace{\phantom{XXxx}}} \nonumber 
\end{align*}
As $F_2$ is a factor (resp. semifactor) set  of $M''$ and $\begin{smat}
B & C \\
0    & E'
\end{smat}$ 
has  linearly independent rows,  then  its complementary    in $M''$   $\begin{smat} 0  & F''\end{smat}$  is FCmR (see Lemma~\ref{RLItoFRC}). Which  is impossible because FCmR ACI-matrices can not have columns full of zeros.

\item Suppose $r_1=0$ and $t_2=0$. Then $M''=M'$ (see~(\ref{M'})). 
As $\begin{smat}
B & C \\
0    & E
\end{smat}$  
has  linearly independent rows then    $F_2$ is not a factor (resp. semifactor) set of $M''$.  Contradiction.

\end{itemize} 
\end{proof}

As we explained before, Theorem~\ref{2semifactor-disjointcase} together with Theorem~\ref{2non-overlappingCupAndCapFS} add up to the following result.
\begin{theorem} \label{cor6}
The intersection and the union of two factor (resp. semifactor) sets of an ACI-matrix are factor (resp. semifactor) sets.
\end{theorem}

\section{The WST-decomposition for ACI-matrices} \label{WST}

 Note  that the set of factor sets of a non FmR ACI-matrix is a partial order set where the order is given by set inclusion. Indeed,  Theorem~\ref{cor6} tells us that  this set  is a lattice. Since it is a finite lattice then it is  bounded. So there is a factor  set that is the maximum or {\bf top factor set}: the union of all factor sets which we will denote $F_{\top}$. And there is another factor  set that is the minimum or {\bf bottom factor set}: the intersection of all factor  sets which we will denote $F_{\bot}$.

The previous paragraph is also valid when we substitute factor sets of a non FmR ACI-matrix by semifactor sets of a FmR ACI-matrix. 

% Note  that the set of factor sets (resp. semifactor sets)  of $A$ is a partial order set where the order is given by set inclusion. {\color{red}Indeed,  Theorem~\ref{cor6} tells us that  it  is a lattice. Since it is a finite lattice then it is  bounded.} So there is a factor (resp. semifactor) set that is the maximum or top: the union of all factor (resp. semifactor) sets which we will denote $F_{\top}$. And there is another factor (resp. semifactor) set that is the minimum or bottom: the intersection of all non-empty factor (resp. semifactor) sets which we will denote $F_{\bot}$.

\begin{theorem} \label{existance&uniqueness}
For any ACI-matrix $M$ there exists a nonsingular $R$ and a permutation matrix $Q$ such that $M$ can be decomposed as follows:
\begin{align} \label{ACI-decompositionABC}
RMQ= \left[\begin{array}{ccc}
\bf{W} & * & *  \\
0 & \bf{S} & * \\ 
0    & 0              &  \bf{T} \\
\end{array}\right]  
\end{align}
where ${\bf W}$ is  a  wide FmR or void, ${\bf S}$ is square FmR or void, and ${\bf T}$ is  a  tall FmR or void. 

Moreover, the ACI-matrices ${\bf W}$, ${\bf S}$ and ${\bf T}$ in decomposition~(\ref{ACI-decompositionABC}) are unique up to equivalence if we impose that ${\bf S}$ is as large as possible  for such a decomposition.
\end{theorem}

\begin{proof}
%We will first show that the decomposition~(\ref{ACI-decompositionABC}) exists. Later on we will see that, in the given decompositions, the ACI-submatrix ${\bf S}$ is as large as possible, and finally we will prove the uniqueness of ${\bf W}$, ${\bf S}$ and ${\bf T}$ up to equivalence.
%When ${\bf S}$ is void in a decomposition of type~(\ref{ACI-decompositionABC}), then we reduce to the case of a factor set or semifactor set. 
Recall that when $M$ is not FmR (resp. $M$ is FmR) then it has at least one factor (resp. semifactor) set. Then this factor (resp. semifactor) set provides a decomposition of type~(\ref{ACI-decompositionABC}) where ${\bf S}$ is void. In this way the existence is solved in a trivial way, but we want to be more demanding and give the  decompositions where the ACI-submatrix ${\bf S}$ is as large as possible because these decompositions will lead to uniqueness. 

%Let $F_{\bot}$ be the intersection of all factor (resp. semifactor) sets and $F_{\top}$ be the union of all factor (resp. semifactor) sets of $M$.  
How do we find such decompositions? Suppose we are given a decomposition as in~(\ref{ACI-decompositionABC}) where ${\bf W}$ is a wide FmR or void, ${\bf S}$ is square FmR or void, and ${\bf T}$ is  a  tall FmR or void.
Let $F_1$ be the set of columns corresponding to ${\bf W}$, and $F_2$ be the set of columns corresponding to $\begin{smat} {\bf W}& * \\ 0 & {\bf S} \end{smat}$. It is easy to check that the $2\times 2$ block partition $\scriptsize\left[\begin{array}{c|cc}
\bf{W} & * & *  \\ \hline
0 & \bf{S} & * \\ 
0    & 0              &  \bf{T} \\
\end{array}\right]$ is an $F_1$-decomposition, and $\scriptsize\left[\begin{array}{cc|c}
\bf{W} & * & *  \\
0 & \bf{S} & * \\  \hline
0    & 0              &  \bf{T} \\
\end{array}\right]$ is an $F_2$-decomposition. Note that the order of $\bf{S}$ corresponds to the difference $\#F_2-\#F_1$.
If we take $F_1=F_{\bot}$ and $F_2=F_\top$ then we will see (Existance) that we obtain a decomposition of type~(\ref{ACI-decompositionABC}). Moreover, as $F_\top$ is the union of all factor (resp. semifactor) sets it is the largest and is unique, and as $F_\bot$ is the intersection of all factor (resp. semifactor) sets it is the smallest and is unique. Since we are taking the extreme sizes, then we will obtain the largest possible order for ${\bf S}$.

%Since  the largest and smallest factor (resp. semifactor) sets exist and are unique, then to obtain the largest possible order of $\bf{S}$ we are forced to make $F_1=F_\bot$ and $F_2=F_\top$, where $F_{\bot}$ is the intersection of all factor (resp. semifactor) sets and $F_{\top}$ is the union of all factor (resp. semifactor) sets. What we will see (Existance) is that we obtain a decomposition of type~(\ref{ACI-decompositionABC}).

% is a factor (resp. semifactor) set, and the columns corresponding to $\begin{smat} {\bf W}& * \\ 0 & {\bf S}\end{smat}$ is also a factor (resp. semifactor) set. 
%
%the properties of given by the hypothesis of the theorem, the sets of columns $F_1$ and $F_2$ defined by
%\begin{align*} 
%& \phantom{X} \overset{F_1}{\overbrace{\phantom{XXXx}}} \nonumber \\[-5mm]
%& 
%\left[\begin{array}{ccc}
%\bf{W} & * & *  \\
%0 & \bf{S} & * \\ 
%0  & 0  &  \bf{T} \\
%\end{array}\right] \hspace{-2mm}  \\[-4mm]
% & \phantom{X} \underset{F_2}{\underbrace{\phantom{i}}} \nonumber 
%\end{align*}
%are factor or semifactor sets.
%Note that the order of $\bf{S}$ corresponds to the difference $\#F_1-\#F_2$.
%Since  the largest and smallest factor or semifactor sets are unique, then to obtain the largest possible order of $\bf{S}$ we are forced to make $F_\top=F_1$ and $F_\bot=F_2$. 

\begin{description}
\item[Existence.]
Let $M$ be an $m\times n$ ACI-matrix. We will make a systematic analysis to be sure that nothing wrong happens even when some of the ACI-submatrices become degenerate or void:
\begin{description}

\item[$M$ is FmR.] We divide the proof into three cases:
\begin{description}
\item[$M$ is  tall.]  As we saw after Definition~\ref{semifactorSetDefinition} the empty set is a semifactor set, so $F_\bot=\emptyset$.  Without loss of generality we can assume that 
$F_\top=\{1,\ldots,k\} $ 
with $0\leq  k \leq  n$. Three subcases are possible:
\begin{enumerate}[(a)]
\item $\emptyset=F_\bot=F_{\top}$. Take ${\bf W}$  void, ${\bf S}$ void, and ${\bf T}=M$ .

\item $\emptyset=F_\bot \subsetneq F_{\top} \subsetneq \{1,\ldots,n\}$.  We do a sweep from bottom to top in $M$ with respect to  $F_\top$ and after reordering the rows we obtain 
\begin{align*}  
& \phantom{XXX} \overset{F_\top}{\overbrace{\phantom{XX}}}  \\[-5mm]
& M\sim \left[\begin{array}{cc}
A & B  \\ 
0  &  C\\
\end{array}\right] \hspace{-2mm}  \nonumber 
\end{align*}
where  $A$ is square with linearly independent rows and $C$ is  tall (see Proposition~\ref{corBig}). And from Lemma~\ref{RLItoFRC}   $A$ is  FmR and $C$ is  FCmR. Take ${\bf W}$ is void, ${\bf S}=A$ and ${\bf T}=C$.

\item $\emptyset=F_\bot \subset F_{\top} = \{1,\ldots,n\}$.  We do a sweep from bottom to top in $M$ with respect to  $F_\top$ and after reordering the rows we obtain
$M\sim \begin{smat}   A  \\  0    \end{smat}$
were  $A$ is square  (see Proposition~\ref{corBig}) and has linearly independent rows. And from Lemma~\ref{RLItoFRC}) $A$ is FmR. Take ${\bf W}$  void, ${\bf S}=A$ and ${\bf T}$ tall degenerate.

\end{enumerate}

\item[$M$ is  wide.] As we saw after Definition~\ref{semifactorSetDefinition} the set $\{1,\ldots,n\}$ is a semifactor set, so $F_\top=\{1,\ldots,n\}$. Without loss of generality we can assume that 
$F_\bot=\{1,\ldots,k\} $ 
with $1\leq  k \leq  n$.  Note that $F_{\bot}=\emptyset$ is not possible  since it will not generate a Medium zero block. Three  subcases are possible:

\begin{enumerate}[(a)]
\item $\emptyset \neq F_\bot = F_{\top} = \{1,\ldots,n\}$. Take ${\bf W}=M$, ${\bf S}$ void and ${\bf T}$ void. 

\item $\emptyset \neq F_\bot \subsetneq  F_{\top} = \{1,\ldots,n\}$ and all the entries of the columns corresponding to $F_{\bot}$  are equal to zero.  Then $M=\begin{smat} 0 & C\end{smat}$ where $C$ is square  (see Proposition~\ref{corBig}) and FmR. Take ${\bf W}$  wide degenerate, ${\bf S}=C$ and ${\bf T}$ void.

\item $ \emptyset \neq F_\bot \subsetneq  F_{\top} = \{1,\ldots,n\}$ and not all the entries of the columns corresponding to $F_{\bot}$  are equal to zero.  We do a sweep from bottom to top in $M$ with respect to  $F_\bot$ and after reordering the rows we obtain 
\begin{align*}  
& M\sim \left[\begin{array}{cc}
A& B  \\ 
0  &  C\\
\end{array}\right] \hspace{-2mm}  \\[-4mm]
 & \phantom{XXx} \underset{\sigma_{F_\bot}(F_\bot)}{\underbrace{\phantom{Xx}}} \nonumber 
\end{align*}
were  $A$ is  wide with linearly independent rows and $C$ is square (see Proposition~\ref{corBig}). And from Lemma~\ref{RLItoFRC}   $A$ is  FRmR and $C$ is  FmR. Take  ${\bf W}=A$, ${\bf S}=C$ and  ${\bf T}$ void.

\end{enumerate}

\item[$M$ is square.] Take ${\bf W}$ void, ${\bf S}=M$ and ${\bf T}$ void. 
\end{description}

\item[$M$ is not FmR.] Then $M$ has factor sets. Note that $F_{\bot}=\emptyset$ is not possible  since it will not generate a Big zero block. Without loss of generality we can assume that 
$$F_\bot=\{1,\ldots,h\} \quad \text{and} \quad F_\top=\{1,\ldots,k\} $$ 
with $0< h \leq k \leq  n$. We consider four cases:

\begin{enumerate}[(1)]
\item $\mathbf{0< h < k <  n}$. We distinguish two possibilities: 

\begin{enumerate}[(a)]
\item Not all the entries of the columns corresponding to $F_{\bot}$  are equal to zero. We do a sweep from bottom to top in $M$ with respect to  $F_\top$ and after reordering the rows we obtain
\begin{align*} 
& \phantom{XXXXXX} \overset{F_\top}{\overbrace{\phantom{XXXx}}} \nonumber \\[-4mm]
& 
M\sim M'=\left[\begin{array}{cc|c}
A & B & C  \\  \cline{1-3}
0    & 0  &  D \\
\end{array}\right] \hspace{-2mm}  \\[-3.5mm]
 & \phantom{XXXXXX} \underset{F_\bot}{\underbrace{\phantom{XX}}} \nonumber 
\end{align*}
where  $\begin{smat}  A & B\end{smat}$ has linearly independent rows. By Lemma~\ref{RLItoFRC}  $\begin{smat}  A & B\end{smat}$ is FRmR and $D$ is FCmR. As $M'$ is obtained  from $M$ without permuting columns, then $F_\bot$ and $F_\top$ are factor sets of $M'$.  
Now we do a sweep from bottom to top in $M'$ with respect to  $F_\bot$ and after reordering the rows we obtain
\begin{align} \label{M''}
& \phantom{XXXXXXX} \overset{F_\top}{\overbrace{\phantom{XXXx}}} \nonumber \\[-4mm]
& 
M' \sim M''=\left[\begin{array}{cc|c}
A'' & B'' & C''  \\
0 & B' & C' \\ \cline{1-3}
0    & 0              &  D \\
\end{array}\right] \hspace{-2mm}  \\[-3.5mm]
 & \phantom{XXXXXXX} \underset{F_\bot}{\underbrace{\phantom{XX}}} \nonumber 
\end{align}
where  $A''$ has linearly independent rows. By Lemma~\ref{RLItoFRC}  $A''$ is FRmR and $\begin{smat}
B' & C' \\ 
0              &  D \\
\end{smat}$ is FCmR. And thus $B'$ is FCmR. 
On the other hand, as $\begin{smat} A & B \end{smat}$ is FRmR and  $\begin{smat} A & B \end{smat} \sim  \begin{smat} A'' & B'' \\ 0  &  B'  \end{smat}$  then $\begin{smat} A'' & B'' \\ 0  &  B'  \end{smat}$ is also FRmR. And thus $B'$ is FRmR. Since $B'$ is FRmR and FCmR then it must be square FmR. Take  ${\bf W}=A''$, ${\bf S}=B'$ and ${\bf T}=D$.

\item  All the entries of the columns corresponding to $F_{\bot}$  are equal to zero.  Then we proceed as in case (1)(a) although now  it is not necessary to perform the second sweep. We finish with $M'=\begin{smat} 0 & B & C \\ 0 & 0 & D \end{smat}$. Take  ${\bf W}$  wide degenerate, ${\bf S}=B$ and ${\bf T}=D$.
\end{enumerate}

\item $\mathbf{0< h < k =  n}$. We distinguish two possibilities: 

\begin{enumerate}[(a)]
\item Not all the entries of the columns corresponding to $F_{\bot}$  are equal to zero. 
The argument is the same as in the case (1)(a) although in  this case the  last  column does  not appear  as $F_{\top}=\{1,\ldots, n\}$.  So  $M'=\begin{smat} A & B \\  0 & 0 \end{smat}$ and  $M''=\begin{smat} A'' & B'' \\ 0 & B' \\ 0 & 0 \end{smat}$. Take ${\bf W}=A''$, ${\bf S}=B'$ and ${\bf T}$  tall degenerate.

\item All the entries of the columns corresponding to $F_{\bot}$  are equal to zero. 
Then we proceed as in case (2)(a) although now  it is not necessary to perform the second sweep.  So  we finish with  $M'=\begin{smat} 0 & B \\ 0 &   0 \end{smat}$. Take ${\bf W}$  wide degenerate, ${\bf S}=B$,  and ${\bf T}$  tall degenerate.

\end{enumerate}

\item $\mathbf{0< h = k <  n}$. We distinguish two possibilities: 

\begin{enumerate}[(a)]
\item Not all the entries of the columns corresponding to $F_{\bot}$  are equal to zero. The argument is the same as (1)(a) although in this case the  second column does not appear as $F_{\bot}=F_{\top}$.  Then only one sweep is necessary. So $M'=\begin{smat} A & C \\ 0 & D \end{smat}$. Take ${\bf W}=A$, ${\bf S}$ void and ${\bf T}=D$.

\item All the entries of the columns corresponding to $F_{\bot}$  are equal to zero. Then $M=\begin{smat} 0 & D\end{smat}$. Take ${\bf W}$  wide degenerate, ${\bf S}$ void,  and ${\bf T}=D$.

\end{enumerate}

\item $\mathbf{0< h = k =  n}$. We distinguish two possibilities: 

\begin{enumerate}[(a)]
\item Not all the entries of the columns corresponding to $F_{\bot}$  are equal to zero. The argument is the same as (3)(a) although in this case the last column does not appear as $F_{\top}=\{1,\ldots, n\}$. Then only one sweep is necessary. So $M'=\begin{smat} A \\ 0 \end{smat}$. Take ${\bf W}=A$, ${\bf S}$ void and ${\bf T}$ tall degenerate.

\item All the entries of the columns corresponding to $F_{\bot}$  are equal to zero. Then $M=\begin{smat} 0 \end{smat}$. Take ${\bf W}$  wide degenerate, ${\bf S}$ void,  and ${\bf T}$  tall degenerate.

\end{enumerate}

\end{enumerate}

\end{description}

\item[Uniqueness up to equivalence of $\bf{W}$, $\bf{S}$ and $\bf{T}$.]
%In this section we will see that the 15 possible decompositions found in the \textbf{Existence} section are actually the unique ones (up to equivalence on ${\bf W}$, $S$ and ${\bf T}$) that have the largest possible $S$. 
We will do the FmR case and the non FmR case together. Actually, we will do all subcases that were studied in the Existence part together, 
since at this point to adapt the general argument to the different subcases should be straightforward (for example, some of the submatrices $P$, $P'$ or $P''$ involved in~(\ref{wsca}) can be void).
%So, depending on the particular subcase that we want to , study,  we will consider those parts of the ACI-matrices that are involved  and those parts that are not (degenerates or void).

So  assume that we have two different decompositions 
 \begin{align*}
& \phantom{XXXXXX} \overset{F_\top}{\overbrace{\phantom{XXXXx}}}  \hspace{53mm} \overset{F_\top}{\overbrace{\phantom{XXXxx}}}   \\[-5mm]
&  R_1MQ_1=\left[\begin{array}{ccc} \mathbf{W}_1 & * & *  \\ 0 & \mathbf{S}_1 & * \\  0 & 0 &  \mathbf{T}_1 \end{array}\right]   \quad \text{ and }\quad   R_2MQ_2=\left[\begin{array}{ccc} \mathbf{W}_2 & * & *  \\ 0 & \mathbf{S}_2 & * \\  0 & 0 &  \mathbf{T}_2 \\ \end{array}\right]\hspace{-2mm}   \\[-4mm]
& \phantom{XXXXXX} \underset{F_\bot}{\underbrace{\phantom{XX}}}  \hspace{63mm} \underset{F_\bot}{\underbrace{\phantom{XX}}} 
\end{align*}

where $R_1$ and $R_2$ are nonsingular matrices, $Q_1$ and $Q_2$ are permutation matrices, $\mathbf{W}_1$ and $\mathbf{W}_2$ are wide FmR or void, $\mathbf{S}_1$ and $\mathbf{S}_2$ are square FmR or void, and $\mathbf{T}_1$ and $\mathbf{T}_2$ are tall FmR or void. Then
$$\begin{matrix}
 \hspace{-20mm} \overset{F_\top}{\overbrace{\phantom{XXXxx}}} \phantom{XXx}   \phantom{XXXXXXs} \overset{F_\top}{\overbrace{\phantom{XXXxx}}}   \\[-4mm]
  \left[\begin{array}{ccc} \mathbf{W}_1 & * & *  \\ 0 & \mathbf{S}_1 & * \\  0 & 0 &  \mathbf{T}_1 \end{array}\right]    =   R_1 R_2^{-1}\left[\begin{array}{ccc} \mathbf{W}_2 & * & *  \\ 0 & \mathbf{S}_2 & * \\  0 & 0 &  \mathbf{T}_2 \\ \end{array}\right] Q_2^{-1} Q_1. \hspace{-2mm} &  \\[-3.5mm]
 \hspace{-30mm}  \underset{F_\bot}{\underbrace{\phantom{XX}}} \phantom{XXXXx}   \phantom{XXXXXXx} \underset{F_\bot}{\underbrace{\phantom{XX}}} 
\end{matrix}$$
Note that the three groups of columns $F_\bot$, $F_\top \setminus F_\bot$ and $\{1,\ldots, n\}\setminus F_\top$ do not change of position. But the columns of each group might get permuted so there are three permutation matrices ($P$ of order $\#F_{\bot}$, $P'$ of order $\#F_{\top}-\#F_{\bot}$, and $P''$ of order $n-\#F_{\top}$) such that $Q_2^{-1} Q_1=\begin{smat} 
P & 0 & 0 \\ 
0 & P' & 0 \\
0 & 0 & P''
\end{smat}$.  So 
\begin{align} \label{wsca}
  \left[\begin{array}{c|cc} \mathbf{W}_1 & * & *  \\ \cline{1-3} 0 & \mathbf{S}_1 & * \\  0 & 0 &  \mathbf{T}_1 \end{array}\right]  
 =   R_1 R_2^{-1}
\left[\begin{array}{c|cc} \mathbf{W}_2 & * & *  \\ \cline{1-3} 0 & \mathbf{S}_2 & * \\  0 & 0 &  \mathbf{T}_2 \\ \end{array}\right] 
\left[\begin{array}{c|cc} 
P & 0 & 0 \\ \cline{1-3}
0 & P' & 0 \\
0 & 0 & P''
\end{array}\right]
\end{align} 
where the lines define  $2\times 2$ block ACI-matrices: we   consider $\begin{smat} \mathbf{S}_1 & * \\ 0 & \mathbf{T}_1 \end{smat}$, $\begin{smat} \mathbf{S}_2 & * \\ 0 & \mathbf{T}_2 \end{smat}$ and $\begin{smat} P' & 0 \\ 0 & P'' \end{smat}$ as just one block. Since $\mathbf{W}_1$ and $\mathbf{W}_2$ are wide FRmR then they have linear independent rows and Lemma~\ref{TwoDecompDivisor2} implies that
\begin{align} \label{a1anda2Etc}
\mathbf{W}_1\sim \mathbf{W}_2 \quad \text{and} \quad \begin{bmatrix} \mathbf{S}_1 & * \\ 0 & \mathbf{T}_1 \end{bmatrix} \sim \begin{bmatrix} \mathbf{S}_2 & * \\ 0 & \mathbf{T}_2 \end{bmatrix}.
\end{align}
In the proof of Lemma~\ref{TwoDecompDivisor2} we saw that $R_1 R_2^{-1}=\begin{smat} R & * \\ 0 & R' \end{smat}$ where $R$ is nonsingular of order $\rows(\mathbf{W}_2)$ and $R'$ is nonsingular of order $\rows(\mathbf{S}_2)+\rows(\mathbf{T}_2)$ with
\[
\begin{bmatrix} \mathbf{S}_1 & * \\ 0 & \mathbf{T}_1 \end{bmatrix} = R' \begin{bmatrix} \mathbf{S}_2 & * \\ 0 & \mathbf{T}_2 \end{bmatrix} \begin{bmatrix} P' & 0 \\ 0 & P'' \end{bmatrix}.
\]
Since $\mathbf{S}_1$ and $\mathbf{S}_2$ are square FmR then they have  linear independent rows and  Lemma~\ref{TwoDecompDivisor2} implies that 
$\mathbf{S}_1\sim \mathbf{S}_2$ and  $\mathbf{T}_1\sim \mathbf{T}_2$.
Which finishs the uniqueness part.
\end{description}
\end{proof}

Since the decomposition of Theorem~\ref{existance&uniqueness} involve a \textbf{W}ide (or void) $\bf{W}$,  a \textbf{S}quare (or void) $\bf{S}$ and  a \textbf{T}all (or void) $\bf{T}$, we will denote this decomposition the \textbf{WST-decomposition} for ACI-matrices whenever $\bf{S}$ is as large as possible.

\section{ACI-matrices of constantRank}

As a application of the WST-decomposition we will refine the main theorem of~\cite[Theorem 5]{MR2775784} by Huang and Zhan, which is a characterization of constantRank ACI-matrices.
The version of the theorem that will be given below makes the degenerate cases much more explicit than the original one. %Another change is that we will use square upper triangular ACI-matrices  with all its diagonal entries equal to 1, while in the original theorem there were square upper triangular ACI-matrices with nonzero constant diagonal entries. It is clear that this change can be done. 
The same version of the theorem was already used in our work on ACI-matrices \cite{BoCa2} as ``\emph{Theorem 2.4 (detailed version)}''.

\begin{theorem} (\cite[Theorem 5]{MR2775784}, see also \cite[Theorem 2.4 (detailed version)]{BoCa2}) \label{HZCharacterization-1}
Let $M$ be a  $m\times n$ ACI-matrix of constantRank $\rho$ with $1 \leq \rho \leq \min \{m, n\}$ over a field  $\mathbb{F}$ with $|\mathbb{F}|\geq \max\{m,n+1\}$. Depending on  $m$, $n$ and $\rho$ we have the following possibilities:
\begin{enumerate}[(i)]

\item $\rho=m<n$ if and only if $ M \sim \begin{smat}
\vspace{-3mm}1 &            &  * &|& \\
\vspace{-1mm}   &\ddots  &     &| &*\\
0 &  & 1 &|& 
\end{smat}$.
\item $\rho=m=n$ if and only if  $M \sim \begin{smat} \vspace{-2mm} 1 & & * \\   & \ddots &  \\ 0 & & 1 \end{smat}.$
\item $\rho=n<m$ if and only if  $ M \sim \begin{smat}
\vspace{-1mm} & * & \\ \cline{1-3}
\vspace{-2mm} 1 &  &  *\\
 &\ddots  & \\
0 &  & 1 
\end{smat}$.

\item $1\leq \rho<\min\{m,n\}$ if and only if for some positive integers $r$ and $s$ with $r+s=m+n-\rho$   
\begin{equation} \label{HZtype}
M \sim 
{\footnotesize \left[\begin{array}{cccc|ccc}
1 &  & \multicolumn{1}{c|}{*} & \multicolumn{1}{c|}{\multirow{3}{*}{\Large $\ * \ $}} &  \multicolumn{3}{c}{\multirow{3}{*}{\Large $*$}} \\ 
  & \ddots & \multicolumn{1}{c|}{} &  \\ 
 0 & & \multicolumn{1}{c|}{1} &  &  & & \\ \hline
 \multicolumn{4}{c|}{\multirow{4}{*}{\Large$0_{r\times s}$}} & \multicolumn{3}{c}{\multirow{1}{*}{\Large $*$}}\\ 
&&&&&& \\
 \cline{5-7}
 & & & & 1 &  & *\\ 
 & & & &  & \ddots &  \\
 & & & & 0 &  & 1 \\ 
\end{array}\right]}
\end{equation}
 where the upper blocks  do not appear if  $r=m$ and   the right blocks do not appear if  $s=n$. 
 \end{enumerate}
\end{theorem}

Finally, we present the refinement of Theorem~\ref{HZCharacterization-1} that we were talking about.

\begin{theorem} \label{maincorollary}
 Let $M$ be an $m\times n$ ACI-matrix $M$ over a field $\mathbb{F}$ such that $|\mathbb{F}|\geq \max\{m,n+1\}$. Then $M$ is constantRank if and only if there exist a nonsingular $R$ and a permutation $Q$ such that
\begin{align}\label{RefinamentTheorem71}
RMQ=
%\begin{bmatrix}
%{\bf W}& * & * \\
%0 & {\bf S}& * \\
%0 & 0 & {\bf T}
%\end{bmatrix}=
{\footnotesize\begin{bmatrix}
\begin{array}{|ccc|c|}\cline{1-4}
1 &  &  * & \\
 &\ddots  & & * \\
0 &  & 1 & \\ \cline{1-4}
\end{array} & * & * \\
0 & \hspace{-2.4ex}\begin{array}{|ccc|}\cline{1-3} 
1 &  & * \\ 
 & \ddots &  \\ 
0 &  & 1 \\ \cline{1-3} \end{array} & * \\
0 & 0 & \hspace{-2.4ex}
\begin{array}{|ccc|}\cline{1-3} 
 & * &  \\ \cline{1-3} 
1 &  & \\ 
 & \ddots &  \\ 
0 &  & 1 \\ \cline{1-3} 
\end{array}
\end{bmatrix}}.
\end{align}
where instead of the ACI-submatrix $\begin{smat}
\vspace{-3mm}1 &            &  * &|& \\
\vspace{-1mm}   &\ddots  &     &| &*\\
0 &  & 1 &|& \\ 
\end{smat}$ there could be a wide degenerate or a void, instead of $\begin{smat}
\vspace{-2mm}1 &  &  *\\
 &\ddots  & \\
0 &  & 1 \\ 
\end{smat}$ there could be a void, and instead of $\begin{smat}
\vspace{-1mm} & * & \\ \cline{1-3}
\vspace{-2mm} 1 &  &  *\\
 &\ddots  & \\
0 &  & 1 \\ 
\end{smat}$ there could be a tall degenerate or a void. If we impose that 
$\begin{smat}
\vspace{-2mm}1 &  &  *\\
 &\ddots  & \\
0 &  & 1 \\ 
\end{smat}$ is as large as possible 
then the three blocks are unique up to equivalence.
\end{theorem}
\begin{proof}
 The sufficiency is obvious. So we proceed with the necessity. Let $R'$ be a nonsingular constant matrix and $Q'$ a permutation matrix such that
% We start applying Theorem~\ref{existance&uniqueness} that implies  that there exist a nonsingular $R'$ and a permutation $Q'$ such that
\begin{align*}
R'MQ'=
\begin{bmatrix}
{\bf W}& * & * \\
0 & {\bf S}& * \\
0 & 0 & {\bf T}
\end{bmatrix}
\end{align*}
is a WST-decomposition of $M$
where: ${\bf W}$ is wide FRmR or void, ${\bf S}$ is square FmR or void, ${\bf T}$ is tall FCmR or void, and ${\bf S}$ is as large as possible. Therefore   
$$\maxRank({\bf W})=\rows({\bf W}),\quad  \maxRank({\bf S})=\rows({\bf S})\quad \text{and} \quad  \maxRank({\bf T})=\cols({\bf T}).$$ 
Moreover,  as $M$ is constantRank then $R'MQ'$  is also constantRank and then  
\begin{align} \label{minRank=maxRankEc1}
\minRank\begin{smat}{\bf W}& * & *\\ 0 & {\bf S}& * \\ 0 & 0 & {\bf T}\end{smat}=\maxRank\begin{smat}{\bf W}& * & *\\ 0 & {\bf S}& * \\ 0 & 0 & {\bf T}\end{smat} =\rows({\bf W})+\rows({\bf S})+\cols({\bf T}).
\end{align}
We will prove that $M$ is equivalent to an ACI-matrix as in (\ref{RefinamentTheorem71}) in three steps. We will assume that ${\bf W}, {\bf S}$ and ${\bf T}$ are not void nor degenerate, otherwise the proof simplifies. 

\begin{enumerate}[(i)]
\item First we will prove that ${\bf W}$, ${\bf S}$ and ${\bf T}$ are full constantRank, that is:
\begin{align} %\nonumber
\minRank({\bf W}) &=\maxRank({\bf W}), \label{minRank=maxRankW}\\
\minRank({\bf S}) &=\maxRank({\bf S}) \text{  and }\label{minRank=maxRankEc-1} \\ 
\minRank({\bf T}) &=\maxRank({\bf T}). \label{minRank=maxRankEc0}
\end{align}

Suppose that $\minRank({\bf T})<\maxRank({\bf T})$. Let $\begin{smat}\widehat{{\bf W}} & * & * \\ 0 & \widehat{{\bf S}} & * \\ 0 & 0 & \widehat{{\bf T}}\end{smat}$ be a completion of $\begin{smat}{\bf W}& * & *\\ 0 & {\bf S}& * \\ 0 & 0 & {\bf T}\end{smat}$ such that 
$\rank(\widehat{{\bf T}})<\maxRank({\bf T})=\cols({\bf T})$.
Then
\begin{align}\label{minRank=maxRankEc2}\nonumber
\minRank\begin{smat}{\bf W}& * & *\\ 0 & {\bf S}& * \\ 0 & 0 & {\bf T}\end{smat}
&\leq \rank\begin{smat}\widehat{{\bf W}} & * & * \\ 0 & \widehat{{\bf S}} & * \\ 0 & 0 & \widehat{{\bf T}}\end{smat}\leq\rows(\widehat{{\bf W}})+\rows(\widehat{{\bf S}})+\rank(\widehat{{\bf T}})\\ &<\rows({\bf W})+\rows({\bf S})+\cols({\bf T}).
\end{align}
As (\ref{minRank=maxRankEc1}) and (\ref{minRank=maxRankEc2}) are contradictory then (\ref{minRank=maxRankEc0}) is true. Similar arguments   prove~(\ref{minRank=maxRankW}) and~(\ref{minRank=maxRankEc-1}).

\item Now, we apply Theorem~\ref{HZCharacterization-1} $(i)$ to ${\bf W}$ to obtain a nonsingular  $R_1$ and a permutation  $Q_1$ such that $R_1 {\bf W}Q_1= 
\begin{smat}
\vspace{-3mm}1 &            &  * &|& \\
\vspace{-1mm}   &\ddots  &     &| &*\\
0 &  & 1 &|& \\ 
\end{smat}$.
We apply Theorem~\ref{HZCharacterization-1} $(ii)$ to ${\bf S}$ to obtain a nonsingular  $R_2$ and a permutation  $Q_2$ such that $R_2 {\bf S}Q_2= \begin{smat}
\vspace{-2mm}1 &  &  *\\
 &\ddots  & \\
0 &  & 1 \\ 
\end{smat}$.
And  we apply Theorem~\ref{HZCharacterization-1} $(iii)$ to ${\bf T}$ to obtain a nonsingular  $R_3$ and a permutation  $Q_3$ such that $R_3 {\bf T}Q_3= \begin{smat}
\vspace{-1mm} & * & \\ \cline{1-3}
\vspace{-2mm} 1 &  &  *\\
 &\ddots  & \\
0 &  & 1 \\ 
\end{smat}$.

\item Finally, if $R:=\begin{smat} R_1 & 0 &  0\\ 0 &R_2  & 0 \\ 0 & 0 & R_3 \end{smat}R'$ and $Q=Q' \begin{smat} Q_1 & 0 &  0\\ 0 &Q_2  & 0 \\ 0 & 0 & Q_3 \end{smat}$ then we obtain the desired result since
\begin{align*}
\small RMQ &= \begin{bmatrix} R_1 & 0 &  0\\ 0 &R_2  & 0 \\ 0 & 0 & R_3 \end{bmatrix}R' M Q' \begin{bmatrix} Q_1 & 0 &  0\\ 0 &Q_2  & 0 \\ 0 & 0 & Q_3 \end{bmatrix}= \begin{bmatrix} R_1 & 0 &  0\\ 0 &R_2  & 0 \\ 0 & 0 & R_3 \end{bmatrix}\begin{bmatrix}
{\bf W}& * & * \\
0 & {\bf S}& * \\
0 & 0 & {\bf T}
\end{bmatrix} \begin{bmatrix} Q_1 & 0 &  0\\ 0 &Q_2  & 0 \\ 0 & 0 & Q_3 \end{bmatrix}\\ 
& = \begin{bmatrix}
R_1 {\bf W}Q_1 & * & * \\
0 & R_2 {\bf S}Q_2 & * \\
0 & 0 & R_3 {\bf T}Q_3 
\end{bmatrix}
\end{align*} 
\end{enumerate}
is an ACI-matrix of type (\ref{RefinamentTheorem71}) where $R_2 {\bf S}Q_2$ is as large as possible.

\end{proof}

The characterization of constantRank ACI-matrices of Theorem~\ref{HZCharacterization-1} has a caveat: there is a restriction on the number of elements of the field that can not be avoided. In~\cite[Theorem 2.5.]{BoCa2} we extended the characterization without any restriction on the field. It is possible to apply the WST-decomposition to refine this extension analogously.

\bigskip

\bigskip

%%%%%%%%%%%%%%%%%%%%%%%%%%%%%%%%%%%%%%%
%%%%%%%%%%%%%%%%%%%%%%%%%%%%%%%%%%%%%%%
%%%%%%%%%%%%%%%%%%%%%%%%%%%%%%%%%%%%%%%

\bibliographystyle{plain}

%%%%%%%%%%%%%%%%%%%%%%%%%%%%%%%%%%%%%%%
%%%%%%%%%%%%%%%%%%%%%%%%%%%%%%%%%%%%%%%
%%%%%%%%%%%%%%%%%%%%%%%%%%%%%%%%%%%%%%%

\end{document}